\documentclass[12pt]{amsart}
\usepackage[left,final]{showlabels}
\usepackage[condensed]{tgheros}

\DeclareRobustCommand{\SkipTocEntry}[5]{}

\usepackage{amssymb,nccmath}
\usepackage{esint}
\usepackage{mathtools}
\usepackage[english]{babel}
\usepackage{epsfig}
\usepackage{mathptmx}
\usepackage{amsmath,amsfonts,amssymb}
\usepackage{mathtools}
\usepackage{enumitem}
\usepackage{mathrsfs}
\usepackage[all]{xy}
\usepackage{graphicx}
\usepackage{latexsym}
\usepackage{verbatim}
\usepackage{xcolor}
\usepackage{xifthen}
\usepackage{cancel}
\usepackage[normalem]{ulem}
\usepackage{accents}

\def\HYPER{\relax}

\ifdefined\HYPER
\usepackage[hidelinks, pdfstartview=FitB, pdfpagemode=UseNone]{hyperref}
\else
\renewcommand{\href}[2]{\relax}
\renewcommand{\url}[1]{#1}
\fi

\usepackage[text={16.5cm, 20cm}, asymmetric]{geometry}
\numberwithin{equation}{section}

\newcommand{\+}{\nobreak\hspace{.087em}\nopagebreak}

\def\Re{\mathop{\mathrm{Re}}}
\def\Im{\mathop{\mathrm{Im}}}

\def\D{\mathbb D}

\def\H{\mathbb{H}}

\def\R{\mathbb R}

\def\C{\mathbb C}
\def\Hol{{\sf Hol}}
\def\Aut{{\sf Aut}(\mathbb D)}

\def\N{\mathbb N}
\def\Q{\mathbb Q}

\def\const{{\rm const}}
\def\id{{\sf id}}

\newtheorem{theorem}{Theorem}[section]
\newtheorem{lemma}[theorem]{Lemma}
\newtheorem{proposition}[theorem]{Proposition}
\newtheorem{corollary}[theorem]{Corollary}

\theoremstyle{definition}
\newtheorem{definition}[theorem]{Definition}
\newtheorem{example}[theorem]{Example}

\theoremstyle{remark}
\newtheorem{remark}[theorem]{Remark}

\numberwithin{equation}{section}

\newcommand{\hd}{\rho}
\newcommand{\phd}{\rho^*}
\newcommand{\phdisk}{\mathrm B_h}

\newcommand{\Zen}{\mathcal{Z}}
\newcommand{\Zn}{\mathcal{Z}_{\scriptscriptstyle\forall}}

\newcommand{\U}{{\sf Uni}}
\newcommand{\UH}{\mathbb{H}}
\newcommand{\UD}{\mathbb{D}}
\newcommand{\UC}{\partial\UD}
\newcommand{\Complex}{\mathbb{C}}
\newcommand{\Real}{\mathbb{R}}
\newcommand{\Natural}{\mathbb{N}}
\newcommand{\di}{\mathop{\mathrm{d}}\nolimits}
\newcommand{\anglim}{\angle\lim}

\renewcommand{\ge}{\geqslant}
\renewcommand{\le}{\leqslant}
\renewcommand{\geq}{\geqslant}
\renewcommand{\leq}{\leqslant}
\newcommand{\proofof}[1]{{\fontseries{bx}\fontshape{it}\selectfont Proof of #1}}

\newcommand{\StepP}[1]{\medskip\noindent{\textsc{Proof of #1.}}}

\newcommand{\StepC}[2]{\medskip\noindent{\textsc{Case}~#1: #2.}}

\newenvironment{ourlist}{\begin{enumerate}[label={\bf (\arabic*)}, ref={\rm (\arabic*)}, left=0em]
\everydisplay{\makeatletter\def\@eqnum{\normalfont(\theequation)}\makeatother}}{\end{enumerate}}
\newenvironment{romlist}{\begin{enumerate}[label={\rm (\roman*)}, ref={\rm (\roman*)}, left=.7em]}{\end{enumerate}}
\newenvironment{statlist}{\begin{enumerate}[label={\bf (\Alph*)}, ref={(\Alph*)}, left=0em]%
\everydisplay{\makeatletter\def\@eqnum{\normalfont(\theequation)}\makeatother}}{\end{enumerate}}

\newcommand{\dff}[1]{\textsl{#1}}

\newcommand\Wtilde[1]{\stackrel{\sim}{\smash{#1}\rule{0pt}{0.54ex}}}
\newcommand{\mtilde}[1]{\mathchoice{\,\widetilde{\!#1\!}\,}{\,\widetilde{\!#1\!}\,}%
{\raise.25ex\hbox{$\scriptstyle\Wtilde{#1}$}}%
{\Wtilde{#1}}}

\begin{document}
\title[Simultaneous linearization and centralizers of parabolic self-maps I]{Simultaneous linearization and centralizers of parabolic self-maps I: zero hyperbolic step}

\author[M. D. Contreras]{Manuel D. Contreras $^\dag$}

\author[S. D\'{\i}az-Madrigal]{Santiago D\'{\i}az-Madrigal $^\dag$}
\address{Camino de los Descubrimientos, s/n\\
	Departamento de Matem\'{a}tica Aplicada II and IMUS\\
	Universidad de Sevilla\\
	Sevilla, 41092\\
	Spain.}\email{contreras@us.es} \email{madrigal@us.es}

\author[P. Gumenyuk]{Pavel Gumenyuk$^\ddag$}\address{Pavel Gumenyuk: Department of
Mathematics\\  Politecnico di Milano, via E. Bonardi 9\\ Milan 20133, Italy.}
\email{pavel.gumenyuk@polimi.it}

\thanks{$^\dag$ Partially supported by  by Ministerio de Innovación y Ciencia, Spain, project PID2022-136320NB-I00.}
\thanks{$^\ddag$ Partially supported  by GNSAGA INdAM (\textit{Istituto Nazionale di Alta Matematica ``Francesco Severi''}) Italy}

\date\today

\begin{abstract} Let $\varphi:\mathbb  D \to \mathbb  D$ be a parabolic self-map of the unit disc~$\mathbb  D$ having \textit{zero} hyperbolic step. We study holomorphic self-maps of~$\mathbb  D$ commuting with~$\varphi$.  In particular, we answer a question from Gentili and Vlacci (1994) by proving that $\psi\in\mathsf{Hol(\mathbb  D,\mathbb  D)}$ commutes with~$\varphi$ if and only if the two self-maps have the same Denjoy\,--\,Wolff point and $\psi$ is a pseudo-iterate of~$\varphi$ in the sense of Cowen. Moreover, we show that the centralizer of~$\varphi$, i.e. the semigroup $\mathcal Z_\forall(\varphi):=\{\psi:\psi\circ\varphi=\varphi\circ\psi\}$ is  commutative. We also prove that if $\varphi$ is univalent, then all elements of~$\mathcal Z_\forall(\varphi)$ are univalent as well, and if $\varphi$ is not univalent, then the identity map is an isolated point of~$\mathcal Z_\forall(\varphi)$. The main tool is the machinery of \textit{simultaneous linearization}, which we develop using holomorphic models for iteration of non-elliptic self-maps originating in works of Cowen and Pommerenke.
\end{abstract}

\maketitle

\tableofcontents

\section{Introduction}\label{Introduction}

One of the main tools in the study of holomorphic dynamical systems is \textsl{linearization}, i.e.,
semiconjugation to a linear map. This paper is a part of a project, which we continue in~\cite{CDG-positive-h.step}, aimed at analyzing and developing the relationship of what is commonly known in Dynamics as
\textsl{\hbox{simultaneous} \hbox{linearization}} with the study of commuting holomorphic self-maps of the unit disc $\UD:={\{z\in\C\colon |z|<1\}}$.

In his seminal papers~\cite{Cowen,Cowen-comm} Cowen introduced what is now known as \textsl{holomorphic \hbox{(semi-)\,models}} for the iteration of holomorphic self-maps of~$\UD$ and used it to study commuting holomorphic self-maps. Obviously, any two iterates of the same self-map commute. Cowen's main result in~\cite{Cowen-comm} can be regarded as a (weaker form of the) converse statement. Namely, he proved that any two commuting holomorphic self-maps are ``generalized iterates'' of a third self-map, i.e. elements of its \textsl{pseudo-iteration semigroup}; see Section~\ref{S_general_facts} for the definition of this notion.

Much later, Arosio and Bracci~\cite{Canonicalmodel,Simultaneous} developed Cowen's ideas and extended them to holomorphic iteration in higher dimensions and, more generally, on complex manifolds. Noticeably, their abstract point of view appears to be fruitful even for holomorphic self-maps of~$\UD$. In this classical case, the notion of a holomorphic semimodel can be interpreted as a way to model the iteration of a given self-map by the iteration of a linear map with the help of a suitable (in general, non-injective) change of variable ${h:\UD\to\C}$; see Section~\ref{Sec:modelos} for more details.

For a non-elliptic (i.e., having  no fixed point in the open disc~$\UD$) holomorphic self-map ${\varphi:\UD\to\UD}$, the linear map modeling~$\varphi$ is of the form $w\mapsto w+1$. Accordingly, ``linearizing'' the non-elliptic  self-map~$\varphi$ consists of studying Abel's functional equation
\begin{equation}\label{EQ_Abel-intro}
 h\circ\varphi=h+1.
\end{equation}
In general, the solution to~\eqref{EQ_Abel-intro} is not unique. However, it is possible to single out one particular solution playing a special role, and often referred to as the \textsl{Koenigs function} of~$\varphi$; see Definition~\ref{DF_canonical}.
In terms of the Koenigs function~$h_\varphi$ of~$\varphi$, one can characterize elements of Cowen's pseudo-iteration semigroup of~$\varphi$ as those holomorphic self-maps ${\psi:\UD\to\UD}$ for which
\begin{equation}\label{EQ_secondAbel-intro}
  h_{\varphi}\circ\psi=h_{\varphi}+c
\end{equation}
holds with a suitable constant ${c\in\Complex}$, which we will refer to as the \textsl{simultaneous linearization coefficient} and which will be denoted by~$c_{\varphi,\psi}$.

From the algebraic point of view, the classes $\Hol(\UD)$ and $\U(\UD)$ consisting all holomorphic and all univalent (i.e. holomorphic and injective) self-maps ${\varphi:\UD\to\UD}$, respectively, are non-abelian semigroups w.r.t. composition. Therefore, the notion of centralizer naturally comes into play.
\begin{definition} Given a holomorphic self-map of the unit disc $\varphi\in \Hol(\D)$, the {\sl centralizer} of $\varphi $ is the set of all holomorphic self-maps of the unit disc which commute with $\varphi,$ that is,
$$
\Zn(\varphi):=\{\psi \in \Hol (\D): \, \psi\circ\varphi=\varphi\circ \psi\}.
$$
When $\varphi$ is also univalent, that is $\varphi\in \U(\D)$, we denote
$$
\Zen(\varphi):=\{\psi \in \U(\D): \, \psi\circ\varphi=\varphi\circ \psi\}=\Zn(\varphi)\cap \U(\D).
$$
\end{definition}

For hyperbolic\footnote{For the classification of holomorphic self-maps in~$\UD$ and for the related notion of the Denjoy\,--\,Wolff point, see Section~\ref{SS_PRE-holo-selfmaps}.} self-maps $\varphi$ the structure of the centralizer $\Zn(\varphi)$
is rather well understood. For parabolic self-maps, which we consider in this paper and in \cite{CDG-positive-h.step}, the situation is much more complicated and open. In particular, there is a significant difference between parabolic self-maps of zero hyperbolic step and those of positive hyperbolic step, as it becomes clear already in the special case of univalent commuting self-maps, which has been recently studied by the authors in~\cite{CDG-Centralizer} in relation to the much celebrated embeddability problem.

For a parabolic self-map $\varphi$ of zero hyperbolic step, in Section~\ref{S_zero} we show (Theorem~\ref{Prop:pseuso-semigroup}) that a self-map $\psi\in\Hol(\UD)$ sharing with~$\varphi$ the same Denjoy\,--\,Wolff point belongs to~$\Zn(\varphi)$ if and only if the Koenigs map~$h_\varphi$ linearizes $\psi$, i.e. if and only if $\psi$ is in the pseudo-iteration semigroup of~$\varphi$. This result was previously known under certain additional conditions, which we are now able to completely eliminate, answering in this way a question going back to Cowen~\cite{Cowen-comm} and explicitly stated by Gentili and Vlacci~\cite{Gentili-Vlacci}, see Remark~\ref{RM_under-additional-condition}.  A closely related result,  see assertion~\ref{IT_THzero-abelian} of Theorem~\ref{Thm:zero-hyperbolic-step}, states that the centralizer $\Zn(\varphi)$ of any parabolic self-map with zero hyperbolic step is \textit{abelian}, i.e. any two elements ${\psi_1,\psi_2\in\Zn(\varphi)}$ commute. In the same theorem, we also show that if $\varphi$ is univalent, then each $\psi\in\Zn(\varphi)$ is univalent as well.

For parabolic self-maps $\varphi$ of positive hyperbolic step, these results do not hold. The structure of~$\Zn(\varphi)$ in the parabolic-positive case can be much richer and will be analyzed in~\cite{CDG-positive-h.step}. In particular, $\Zn(\varphi)$ is not necessarily abelian even in the univalent case (see \cite[Example 8.5]{CDG-Centralizer}) and, in general, is not contained in the pseudo-iteration semigroup.

Returning to the parabolic-zero case, we further show (Theorem~\ref{Thm:homeomorphism}) that the map assigning to each element of the centralizer ${\psi\in\Zn(\varphi)}$ its simultaneous linearization coefficient~$c_{\varphi,\psi}$ is an isomorphism between the topological semigroup~$\Zn(\varphi)$ and a subsemigroup of~$\C$ regarded as a topological semigroup w.r.t. addition. This allows us to establish another remarkable fact (Theorem~\ref{Thm:embedding}):
if the identity map~$\id_{\D}$ is not isolated in the centralizer~$\Zn(\varphi)$, then it contains a non-trivial continuous one-parameter semigroup~$(\phi_t)_{t\ge0}$ of holomorphic self-maps of the unit disc and then the function $\varphi$ is needfully univalent. If in addition, the self-map~$\varphi$ has a boundary fixed point different from its Denjoy\,--\,Wolff point, then in combination with \cite[Theorem~1.4]{CDG-Embedding}, this result implies that $\varphi$ is contained in the semigroup~$(\phi_t)$ and hence, $\varphi$~is embeddable.

In Section~\ref{S_SLC}, we prove an explicit formula for the simultaneous linearization coefficient\footnote{Previously, this formula was found by the authors for univalent parabolic self-maps in~\cite{CDG-Centralizer}.} (see~\eqref{EQ_SLC} in Theorem~\ref{TH_SLC(2formulas)}), which allows one to determine~$c_{\varphi,\psi}$ directly, i.e. without knowing the Koenigs map of~$\varphi$. We also study the uniqueness questions for simultaneous linearization. For a parabolic self-map $\varphi$ of zero hyperbolic step and any ${\psi\in\Zn(\varphi)}$, we show (Theorem~\ref{TH_SLC(2formulas)}\+\ref{IT_SLC2}) that if a holomorphic map ${h:\UD\to\C}$ linearizes both $\varphi$ and~$\psi$, i.e. ${h\circ\varphi}={h+c_1}$ and ${h\circ\psi}={h+c_2}$ with ${|c_1|^2+|c_2|^2\neq0}$, then ${c_{\varphi,\psi}=c_2/c_1}$. Furthermore, if $c_{\varphi,\psi}$ is not a rational real number, then $h$ necessarily coincides, up to an affine transformation, with the Koenigs map~$h_\varphi$ of~$\varphi$, see Proposition~\ref{PR_SLC(cont)}.  Finally, we establish sufficient conditions on the hyperbolic step and the simultaneous linearization coefficient, under which the commutativity of parabolic self-maps $\varphi$ and~$\psi$ implies that the Koenigs map~$h_\psi$ of~$\psi$ is a multiple of~$h_\varphi$; see Theorem~\ref{TH_step-and-Koenigs}.

The rest of the paper is organized as follows. In Section~\ref{S_prelim}, we introduce the necessary preliminaries on the dynamics of holomorphic self-maps of~$\UD$. Moreover, in Subsection~\ref{Sec:modelos} we prove a few auxiliary results, used in the subsequent sections.

In Section~\ref{S_general_facts}, we explain the relationship between the pseudo-iteration semigroup of a parabolic self-map and simultaneous linearization. We show (Proposition~\ref{PR_affine-part}) that under a light additional condition, the simultaneous linearizability of two parabolic self-maps with the same Denjoy\,--\,Wolff point implies that the self-maps commute. In the same section, we also analyse the case of \textit{real} simultaneous linearization coefficient (Proposition~\ref{Prop:firstpropertiescoefficient}).

The paper is concluded in Section~\ref{S_APP} serving as an appendix, in which we provide a necessary and sufficient condition  in terms of the hyperbolic derivative for a non-elliptic self-map to be of zero hyperbolic step.

\section{Preliminaries and auxiliary results}\label{S_prelim}
Below we introduce some notation and basic theory used further in the paper. For  more details and for the proofs of the previously known results presented in this section without proof, we refer the interested readers to the recent monograph~\cite{Abate2}.

\addtocontents{toc}{\SkipTocEntry}
\subsection{Notation}\label{Notation}
As usual, we denote the unit disc by
${\UD:=\{z\in\C:|z|<1\}}$,
and we write $\UH:={\{w\in\C:\Im w>0\}}$ for the upper half-plane and $\H_{\mathrm R}:={\{w\in\C:\Re w>0\}}$ for the right half-plane.

For any two sets $A$ and $B$, the inclusion $A\subset B$ will be understood in the wide sense, i.e. allowing the equality~${A=B}$ as a special case.

Furthermore, denote by $\Hol(D,E)$ the class of all holomorphic mappings of a domain $D\subset\C$ into a set $E\subset\C$,
and let $\U(D,E)$ stand for the class of all \textit{univalent} (i.e. injective holomorphic) mappings from $D$ to~$E$.  As usual, we endow $\Hol(D,E)$ and $\U(D,E)$ with the topology of locally uniform convergence. In case $E=D$, we will write  $\Hol(D)$ and $\U(D)$ instead of $\Hol(D,D)$ and $\U(D,D)$, respectively.

For a self-map $\varphi:D\to D$ of a domain $D\subset\C$ and ${n\in\Natural}$ we denote by $\varphi^{\circ n}$ the $n$-th iterate of~$\varphi$, and let $\varphi^{\circ0}:=\id_D$, the identity map in $D$. Moreover, if $\varphi$ is an automorphism of~$D$, then for every $n\in\N$, we denote by $\varphi^{\circ-n}$ the $n$-th iterate of~$\varphi^{-1}$.

If $D$ is a hyperbolic domain in the complex plane, we denote by $\hd_{D}$ (resp. $\phd_D$) the hyperbolic distance (resp. pseudohyperbolic distance) in $D$. We write $\phdisk^{D}(z,r)$ for the pseudohyperbolic disc in~$D$ of radius~$r$ centered at~$z$. When $D$ is the with disc, we simply write $\phdisk(z,r)$ to denote $\phdisk^{\D}(z,r)$.

\addtocontents{toc}{\SkipTocEntry}
\subsection{Holomorphic self-maps of the unit disc}\label{SS_PRE-holo-selfmaps}
The study of the dynamics of a generic holomorphic self-map $\varphi$ of the unit disc $\mathbb{D}$, different from the identity map,  is a classical and well-established branch of Complex Analysis.

The central result in the area is the Denjoy\,--\,Wolff Theorem, which states that if $\varphi$ is different from an elliptic automorphism (i.e. not an automorphism of~$\UD$ possessing a fixed point in~$\UD$), then the sequence of the iterates $(\varphi ^{\circ n})$ converges locally uniformly in~$\UD$ to a certain point~${\tau\in\overline{\mathbb{D}}}$.  This point  is called the \dff{Denjoy\,--\,Wolff point\/} of $\varphi$. Moreover, if $\tau\in \partial \D$, it is the unique boundary fixed point at which the angular derivative $\varphi'(\tau)$ is finite and belongs to $(0,1]$.

According to the position of the Denjoy\,--\,Wolff point~$\tau$ and to the value of the \textsl{multiplier} $\varphi'(\tau)$, holomorphic self-maps $\varphi\in\Hol(\UD)$ different from elliptic automorphisms are divided into three categories. Namely, $\varphi$ is called:
\begin{itemize}
\item[(a)] \dff{elliptic\/} if $\tau\in\UD$,

\item[(b)] \dff{hyperbolic\/} if $\tau\in \partial \D$ and $\varphi'(\tau )<1$, and

\item[(c)] \dff{parabolic\/} if $\tau
\in \partial \D$ such that $\varphi'(\tau )=1$.
\end{itemize}
All elliptic automorphisms of~$\UD$, along with the identity mapping~$\id_\UD$, are conventionally included in the category~(a) of elliptic self-maps.
Similarly, for an elliptic automorphism different from~$\id_\UD$, its Denjoy\,--\,Wolff point is defined to be its unique fixed point in~$\UD$.

Parabolic self-maps can have very different dynamical properties depending on the so-called \textit{hyperbolic step}.
Let $\varphi\in\Hol(\UD)$ be non-elliptic. Thanks to the Schwarz\,--\,Pick Lemma, for the orbit $\big(z_n\big):=\big(\varphi^{\circ n}(z_0)\big)$ of any point ${z_0\in\UD}$, there exists a finite limit $q(z_0):=\lim_{n\to+\infty} \hd_\D(z_{n},z_{n+1})$.
 It is known, see e.g. \cite[Corollary\,4.6.9]{Abate2}, that  either $q(z_0)>0$ for all~${z_0\in\UD}~$, or $q\equiv0$ in~$\UD$.  The self-map~$\varphi$ is said to be of \dff{positive} or of \dff{zero hyperbolic step} depending on whether the former or the latter alternative occurs.
If $\varphi$ is
hyperbolic, then it is always of positive hyperbolic step. However, there exist parabolic self-maps of zero as well as of positive hyperbolic step.

\addtocontents{toc}{\SkipTocEntry}
\subsection{Commuting holomorphic self-maps}\label{SS_commuting}
It is clear that if two  holomorphic self-maps $\varphi,\psi\in\Hol(\UD)\setminus\{\id_\UD\}$ commute, i.e. ${\varphi\circ\psi}={\psi\circ\varphi}$, and if one of them is elliptic, then the other is also elliptic and they share the Denjoy\,--\,Wolff point. The situation is not so evident when we consider non-elliptic self-maps. In 1973, Behan \cite{Behan}, see also \cite[Section~4.10]{Abate2}, proved that if $\varphi$ and~$\psi$ are commuting non-elliptic self-maps of~$\UD$ and $\varphi$ is not a hyperbolic automorphism, then  $\varphi$ and~$\psi$ share the  Denjoy\,--\,Wolff point.

Later, Cowen~\cite[Corollary~4.1]{Cowen-comm} proved that if $\varphi$ and $\psi$ are two non-elliptic commuting holomorphic self-maps of~$\D$ and if $\varphi$  is hyperbolic, then $\psi$ is also hyperbolic (and thus, if $\varphi$ is parabolic, then $\psi$ is parabolic)\footnote{See also \cite[Theorem~1.3]{Simultaneous} for a generalization of this result to holomorphic self-maps of the unit ball in~$\C^n$.}. In the special case when $\varphi$ is a hyperbolic automorphism, by an old result of Heins \cite[Lemma~2.1]{Heins}, $\psi\in\Hol(\UD)\setminus\{\id_\UD\}$ commutes with~$\varphi$ if and only if $\psi$ is a hyperbolic automorphism itself with the same fixed points as~$\varphi$.

It is worth mentioning that parabolic self-maps of positive hyperbolic step can commute with parabolic self-maps of zero hyperbolic step. Moreover, in contrast to the hyperbolic case, the fact that one of them is an automorphism would not imply that the other must be also an automorphism. To see this, one can consider the following example: ${\varphi:=h^{-1}\circ(h+1)}$, ${\psi:=h^{-1}\circ(h+i)}$, where $h$ is a conformal map of~$\UD$ onto~$\UH$.

\addtocontents{toc}{\SkipTocEntry}
\subsection{Holomorphic models for holomorphic self-maps}\label{Sec:modelos}
An indispensable role in our study is played by the concept of a holomorphic model, which goes back to Pommerenke~\cite{Pom79}, Baker and Pommerenke~\cite{BakerPommerenke}, and Cowen~\cite{Cowen}. The terminology we use is mainly borrowed from~\cite{Canonicalmodel}.

\begin{definition}\label{DF_absirbing} Let  $\varphi\in\Hol(\D)$. A domain $V \subset \D$ is \emph{invariant} for $\varphi$ (or $\varphi$-invariant) if $\varphi(V)\subset V$; it is \emph{absorbing} for $\varphi$ (or $\varphi$-absorbing) if it is $\varphi$-invariant and
$$
\D=\bigcup_{n\in\Natural}\big(\varphi^{\circ n}\big)^{-1}(V).
$$
In other words, a $\varphi$-invariant domain is $\varphi$-absorbing if it eventually contains the orbit of any point of $\D$.
\end{definition}

\begin{definition}\label{DF_holomorphic-model} A \dff{holomorphic semimodel} of $\varphi\in\Hol(\D)$  is a triple $\mathcal M:=(S,h,\alpha)$, where $S$ is a Riemann surface, $\alpha$ is an automorphism of $S$, and $h$ is a holomorphic map from $\UD$ into $S$ satisfying the following two conditions:
	\begin{enumerate}[left=2.5em]
		\item[(HM1)] $h\circ \varphi=\alpha \circ h$  {}~and
		\item[(HM2)] $S\,=\,\bigcup_{n\geq0}  \big(\alpha^{\circ n}\big)^{-1}(h(\D))$.
	\end{enumerate}
		A holomorphic \dff{model} of $\varphi\in \Hol (\D)$ is a semimodel $\mathcal M:=(S,h,\alpha)$ for which
	\begin{enumerate}[left=2.5em]
		\item[(HM3)] there exits a $\varphi$-absorbing domain $V\subset\UD$ in which $h$ is injective.
	\end{enumerate}
The Riemann surface $S$ is called the \dff{base space}, and the map $h$ is called the \dff{intertwining map} of the holomorphic model~$\mathcal M$.
\end{definition}

Given a holomorphic model $\mathcal M$ of~$\varphi$, every holomorphic semimodel of~$\varphi$ factorises via the model~$\mathcal M$ in the following sense.

\begin{lemma}[{\cite[Lemma 3.5.8 and Remark 3.5.6]{Abate2}}]\label{Le:morphism}
Assume that $\varphi\in\Hol(\UD)$ admits a holomorphic model $\mathcal M:=(S,h,\alpha)$.
If $(S_1,h_1,\alpha_1)$ is another holomorphic semimodel for~$\varphi$, then there exists  a surjective holomorphic map $\beta:S\to S_1$ such that ${h_1=\beta\circ h}$ and ${\beta\circ\alpha}={\alpha_1\circ\beta}$.
\end{lemma}

Every
non-elliptic self-map $\varphi\in\Hol(\UD)$ admits an essentially unique holomorphic model. More precisely, the following fundamental theorem holds.

\begin{theorem}[\protect{\cite[Corollary~4.5.5]{Abate2}}] \label{Thm:uniqness} Every
non-elliptic self-map $\varphi\in\Hol(\UD)$ admits a holomorphic model. Moreover such a model is unique up to a model isomorphism; i.e., if $(S_1,h_1,\alpha_1)$ and $(S_2,h_2,\alpha_2)$ are holomorphic models for $\varphi$, then there exists a biholomorphic map $\beta$ of~$S_1$ onto~$S_2$ such that
	$$
	h_2=\beta\circ h_1,\quad \alpha_2=\beta\circ\alpha_1\circ \beta^{-1}.
	$$
\end{theorem}

The proof of the existence of holomorphic models depends strongly on the existence of absorbing sets (see, i.e., \cite[Theorem 3.5.10]{Abate2}).
In the next result we summarize what we need about the models for self-maps of the unit disc.
The type of a self-map is reflected in, and actually can be fully determined from, the kind of holomorphic model $\varphi$ admits.  For an open interval $I\subset\Real$, we define
$$
 S_I:=\Real\times I=\{x+iy:x\in\Real,\,y\in I\}.
$$

\begin{theorem}[\protect{\cite{Cowen}, see also \cite[Theorem 4.6.8]{Abate2}}]\label{Thm:model}
Let $\varphi\in\Hol(\UD)$. The following statements hold.
\begin{ourlist}
	\item\label{IT_HM-hyp} $\varphi$ is a hyperbolic self-map with multiplier $\lambda\in(0,1)$ if and only if $\varphi$ admits a holomorphic model of the form $\mathcal M_\varphi:=(S_{I},h,z\mapsto z+1)$, where $I=(a,b)$ is a bounded open interval of length ${b-a=\pi/|\log\lambda|}$.
	\item\label{IT_HM-para-PHS} $\varphi$ is a parabolic self-map of positive hyperbolic step if and only if $\varphi$ admits a holomorphic model of the form ${\mathcal M_\varphi:=(S_{I},h,z\mapsto z+1)}$, where $I$ is an open unbounded interval different from the whole~$\Real$.
	\item\label{IT_HM-para-0HS} $\varphi$ is a parabolic self-map of zero hyperbolic step if and only if $\varphi$ admits a holomorphic model of the form ${\mathcal M_\varphi:=(\C,h,z\mapsto z+1)}$.	
\end{ourlist}	
\end{theorem}

\begin{remark}\label{RM_normalization}
In the above theorem, we may assume that:
\begin{itemize}
 \item[-]in cases~\ref{IT_HM-hyp} and~\ref{IT_HM-para-0HS}, $h(0)=0$;
 \item[-]in case~\ref{IT_HM-para-PHS}, $\Re h(0)=0$ and ${S_I=S_{(0,+\infty)}=\UH}$ or ${S_I=S_{(-\infty,0)}=-\UH}$.
\end{itemize}
Using the uniqueness part of Theorem~\ref{Thm:uniqness}, one can show (see e.g. \cite[Corollary 4.6.12]{Abate2} for details) that the above assumptions play the role of a normalization under which the holomorphic model $\mathcal M_\varphi$ for a given non-elliptic self-map~$\varphi$ is unique. Note that the normalization for cases \ref{IT_HM-hyp} and~\ref{IT_HM-para-0HS} would also work in case~\ref{IT_HM-para-PHS}, but we prefer to use another normalization, so that for parabolic self-maps of positive hyperbolic step,  the base space~$S_I$ of~$\mathcal M_\varphi$ coincides with $\UH$ or~$-\UH$. Moreover, replacing, if necessary, $\varphi$ with $z\mapsto\overline{\varphi(\bar z)}$ we may assume in most of the proofs that ${S_I=\UH}$.
\end{remark}

\begin{definition}\label{DF_canonical}
 The unique holomorphic model $\mathcal M_\varphi$ of a non-elliptic self-map  $\varphi\in\Hol(\UD)\setminus\{\id_\UD\}$ defined in Theorem~\ref{Thm:model} and normalized as in Remark~\ref{RM_normalization} is called \dff{the canonical (holomorphic) model} for~$\varphi$. The intertwining map~$h$ of the canonical model~$\mathcal M_\varphi$ is called the \dff{Koenigs function}, and we denote it, from now on, $h_{\varphi}$. Finally, ${h_{\varphi}(\UD)}$ is called the \dff{Koenigs domain} of~$\varphi$.
\end{definition}

\begin{remark} \label{Re:almostasemigroup}
Let $\varphi\in\Hol(\UD)$ be a non-elliptic self-map with canonical holomorphic model  $(S,h_{\varphi},z\mapsto z+1)$. By the very definition of holomorphic model, we have that for any compact set $K$ in $S$, there exists  $N>0$ such that for all natural number $n>N$, it holds that $K+n\subset h_{\varphi}(\D)$. In fact, something stronger can be proved: Take $K$ a compact set in $S$, and consider the compact set $$\mtilde K=\{w+s:\, w\in K, s\in [0,2]\}\subset S.$$ Then there is $N>0$ such that $\mtilde K+n \subset h_{\varphi}(\D)$ for all  $n>N$. Therefore,  ${K+t\subset h_{\varphi}(\D)}$ for any real number ${t>N}$. Roughly speaking, this means that asymptotically every Koenigs domain behaves like the Koenigs domain of a non-elliptic semigroup. This fact has been used several times in the literature, see e.g.  \cite[Lemma~7.6]{Bracci-Roth} and   \cite[Lemma~2.2]{Poggi}.
A bit more generally, if ${V\subset\UD}$ is any $\varphi$-absorbing open set, then  repeating the same argument with $\UD$ replaced by~$V$, we see that for each compact $K\subset S$ there exists $t_0={t_0(K,V)\ge0}$ such that ${K+t\subset h_{\varphi}(V)}$ for all ${t\in[t_0,+\infty)}$.
\end{remark}

The following result proved in~\cite{CDP} plays an important role in our study. For  a given parabolic self-map $\varphi\in\Hol(\UD)$, it provides a construction of a family of invariant sets with certain useful properties. In particular, if $\varphi$ is of zero hyperbolic step then these sets are $\varphi$-absorbing. Recall that a domain $U\subset\UD$ is said to be isogonal at a point~${\sigma\in\UC}$ if for any Stolz angle ${S\subset\UD}$ with vertex at~$\sigma$ there exists~${\varepsilon>0}$ such that ${\{z\in S:|z-\sigma|<\varepsilon\}\subset U}$.

\begin{theorem} [\protect{\cite[Theorem 5.1]{CDP}}]
\label{Thm:V-sets}
Let $\varphi \in \Hol(\D)$ be
parabolic  with Denjoy\,--\,Wolff point $\tau\in \partial \D$ and Koenigs function $h_{\varphi}$.
Given $0<r <1,$ write\begin{equation*}
V_{\varphi}(r )=\left\{ z\in \mathbb{D}:\phd_\UD(z,\varphi
(z))<r \right\}.
\end{equation*}Then,
\begin{ourlist}
\item\label{IT_V-sets(1)}
     $V_{\varphi}(r )$ is  $\varphi$-invariant, that is, $\varphi (V_{\varphi}(r ))\subset V_{\varphi}(r )$ for all $0<r <1.$

\item\label{IT_V-sets(2)}
     $V_{\varphi}(r)$ is a simply connected domain.

\item\label{IT_V-sets(3)} $V_{\varphi}(r )$ is isogonal at $\tau .$

\item\label{IT_V-sets(4)}
    If $0<r \leq \frac{1}{3},$ then both $\varphi $ and its  Koenigs map $h_{\varphi}$ are univalent in $V_{\varphi}(r ).$

\item\label{IT_V-sets(5)}
     If $\varphi $ is of zero hyperbolic step and $0<r \leq \frac{1}{3}$, then $V_{\varphi}(r )$ is $\varphi$-absorbing.
\end{ourlist}
\end{theorem}

The notation $V_{\varphi}(r)$ introduced in the above theorem will be used frequently throughout the paper. In addition, some elementary properties of these sets will be used and, for the sake of clearness, we collect them in the next lemma.

\begin{lemma}\label{Le:V-sets}
Let $\varphi ,\psi \in \Hol(\D)$ be parabolic. The following statements hold.
\begin{ourlist}
\item\label{V-sets-fact1} If $\psi\in  \Zn(\varphi)$, then $V_{\varphi}(r)$ is $\psi$-invariant for all ${r\in(0,1)}$.
\item\label{V-sets-fact2} Take $z,w\in \D$ such that $\phd_{\D}(z,\psi(z))<a$ and $\phd_{\D}(z,w)<b$. Then $\phd_{\D}(w,\psi(w))<a+2b$.
\item\label{V-sets-fact3} If $z \in V_{\psi}(1/9)$ and $\phd_{\D}(z,w)<1/9$, then $w\in V_{\psi}(1/3)$. In other words, $\phdisk(z,1/9)\subset V_{\psi}(1/3)$ whenever ${z \in V_{\psi}(1/9)}$.
\item\label{V-sets-fact4a} If ${z\in V_{\psi}(1/9)}$, then $\phdisk\big(\psi^{\circ n}(z),1/9\big)\subset V_{\psi}(1/3)$ for any~${n\in\Natural}$.
\item\label{V-sets-fact4} If $\psi\in  \Zn(\varphi)$ and ${z\in V_{\psi}(1/9)}$, then $\phdisk\big(\varphi^{\circ n}(z),1/9\big)\subset V_{\psi}(1/3)$ for any~${n\in\Natural}$.
\end{ourlist}
\end{lemma}
\begin{proof}
If  $\psi\in  \Zn(\varphi)$ and $z\in V_\varphi (r)$, then
$$
\phd_{\D}\big(\psi(z), \varphi\circ \psi(z)\big)=\phd_{\D}\big(\psi(z), \psi\circ \varphi(z)\big)\leq \phd_{\D}\big(z,  \varphi(z)\big)<r,
$$
so that $\psi(z)\in V_\varphi (r)$, where we have used the Schwarz\,--\,Pick Lemma. Thus \ref{V-sets-fact1} holds. Further, \ref{V-sets-fact2} is a consequence of the triangular inequality and again the Schwarz\,--\,Pick Lemma, and \ref{V-sets-fact3} follows from~\ref{V-sets-fact2}.

By Theorem~\ref{Thm:V-sets}\+\ref{IT_V-sets(1)},  ${\psi^{\circ n}(z)\in V_{\psi}(1/9)}$ for any ${z\in V_{\psi}(1/9)}$ and any~${n\in\Natural}$. Therefore, \ref{V-sets-fact3} implies~\ref{V-sets-fact4a}.
Similarly, by~\ref{V-sets-fact1} with $\varphi$ and~$\psi$ interchanged, ${\varphi^{\circ n}(z)\in V_{\psi}(1/9)}$ for any ${z\in V_{\psi}(1/9)}$ and any~${n\in\Natural}$; as a result, \ref{V-sets-fact4} follows from~\ref{V-sets-fact3}.
\end{proof}

The nice properties of the sets $V_\varphi(r)$ allow us to show that the type of the self-map $\varphi$ and the type of the restriction  $\varphi|_{V_\varphi(r)}$ coincide.  More precisely, we are going to prove the following statement.

Fix a parabolic self-map $\varphi\in\Hol(\D)$ and some $r\in(0,1/3]$. Let $f$ be a conformal map of~$\UD$ onto~${V_\varphi(r)}$. Since ${V_\varphi(r)}$ is $\varphi$-invariant, $\mtilde \varphi:={f^{-1}\circ \varphi\circ f}$ is a holomorphic self-map of~$\UD$.

\def\Mdd{\mathcal M_{\mtilde\varphi}}
\def\Md{\mathcal M_{\varphi}}
\begin{proposition} \label{modelV_par}
In the above notation, for any parabolic self-map $\varphi\in\Hol(\D)$ the following statements hold:
	\begin{ourlist}
     \item\label{IT_modelV-zero-step}
           If $\varphi$ is of zero hyperbolic step with canonical model ${\Md:=(\C, h_\varphi,z\mapsto z+1)}$, then $${\Mdd:=(\C, h_\varphi\circ f,z\mapsto z+1)}$$ is a holomorphic model of $\mtilde \varphi$ and thus, $\mtilde \varphi$ is also parabolic of zero hyperbolic step. \medskip

	 \item\label{IT_modelV-pos-step}
           If $\varphi$ is  of positive hyperbolic step with canonical model $\Md:={(S, h_\varphi\circ f,z\mapsto z+1)}$, then $\mtilde\varphi$ is also parabolic of positive hyperbolic step. Moreover, if ${S=\H}$ (resp. $S=-\H$), then $\Mdd:={(\Pi_r, h_\varphi\circ f,z\mapsto z+1)}$ (resp. $\Mdd:={(-\Pi_r, h_\varphi\circ f,z\mapsto z+1)}$), where
		   $$
		    \Pi_r:=\big\{w\in\H:\phd_\H(w,w+1)<r\big\}=\big\{w\in\H:\Im w>\sqrt{1-r^2\,}\big/(2r)\big\},
		   $$
is a holomorphic model of~$\mtilde \varphi$.
	\end{ourlist}
\end{proposition}
\begin{proof}For brevity, we denote $V_r:=V_\varphi(r)$.

\StepP{\ref{IT_modelV-zero-step}}	
	Note that
	$$
	h_\varphi\circ f\circ\mtilde \varphi=h_\varphi\circ\varphi\circ f=h_\varphi\circ f +1.
	$$
	Hence, using $f(\D)=V_r$, we have that $\Mdd:={(S, h_\varphi\circ f,z\mapsto z+1)}$ is a semimodel of $\mtilde \varphi$, where $S:=\bigcup_{n=0}^\infty (h_\varphi(V_r)-n)$.
	Take $w\in\C$. By the absorbing property of the model, there exist $z\in\D$ and $n_0\in\N$ such that $w=h_\varphi(z)-n_0$. By Theorem~\ref{Thm:V-sets}\+\ref{IT_V-sets(5)}, $V_r$ is $\varphi$-absorbing. Hence, there exists ${m_0\in\N}$ such that ${\varphi^{\circ m_0}(z)\in V_r}$. Therefore,
		$$
		w=h_\varphi(z)+m_0-(n_0+m_0)=h_\varphi(\varphi^{\circ m_0}(z))-(n_0+m_0)\in \bigcup_{n=0}^\infty (h_\varphi(V_r)-n)=S.
		$$
Thus $S=\C$. By Theorem~\ref{Thm:V-sets}\+\ref{IT_V-sets(4)}, $h_\varphi$ is univalent in~$V_r$. Hence, ${h_\varphi\circ f}$ is univalent in~$\UD$ and as a result, the semimodel~$\Mdd$ is actually a holomorphic model for~$\mtilde\varphi$. By Theorem~\ref{Thm:model}\+\ref{IT_HM-para-0HS}, it follows that~$\mtilde\varphi$ is parabolic of zero hyperbolic step, as desired.

\StepP{\ref{IT_modelV-pos-step}}
	We assume that the base space of the canonical model of $\varphi$ is the upper half-plane~$\H$. The proof in other case, that is, when the base domain is~$-\H$, is completely similar.

 Take $w\in\H$. Then,
	$$
	\phd_\H(w,w+1)=\left|\dfrac{w-(w+1)}{w-\overline{(w+1)}}\right|=\left|\dfrac{1}{-1+2i\Im w}\right|=\frac{1}{\sqrt{1+4\Im^2w}}.
	$$
Therefore, $w\in \Pi_r$ if and only  $\Im w>\frac12\sqrt{(1-r^2)/r^2\,}.$

Let us show that
\begin{equation}\label{EQ_Pi-absorbing}
	\Pi_r=\bigcup_{n=0}^\infty (h_\varphi(V_r)-n).
\end{equation}
	Firstly, take $w\in\bigcup_{n=0}^\infty (h_\varphi(V_r)-n)$. Then $w=h_\varphi(z)-n$ for some ${z\in V_r}$ and some ${n\in\N}$.
	It follows that
	$$ \phd_\H(w,w+1)=\phd_\H\big(h_\varphi(z)-n,h_\varphi(z)-n+1\big)=\phd_\H\big(h_\varphi(z),h_\varphi(\varphi(z))\big) \leq \phd_\D\big(z,\varphi(z)\big)<r,
	$$
	where we used the invariant Schwarz\,--\,Pick Lemma, see e.g. \cite[Theorem 6.4]{Beardon-Minda}.
	Hence, ${w\in\Pi_r}$.

Reciprocally, take $w\in\Pi_r$. By the absorbing property of the model, see~(HM2) in Definition~\ref{DF_holomorphic-model}, there exist ${z\in\D}$ and ${n_0\in\N}$ such that $w=h_\varphi(z)-n_0$. Therefore, by \cite[Lemma 3.16]{Canonicalmodel},
$$
\lim_{m\to\infty}
	\phd_\D\big(\varphi^{\circ m}(z),\varphi^{\circ (m+1)}(z)\big)=
        \phd_\H\big(h_\varphi(z),h_\varphi(\varphi(z))\big)=
           \phd_\H\big(h_\varphi(z)-n_0,h_\varphi(z)-n_0+1\big) < r.
	$$
	Therefore, there exists $m_0\in\N$ such that $\phd_\D\big(\varphi^{\circ m_0}(z),\varphi^{\circ (m_0+1)}(z)\big)<r$. Thus $\varphi^{\circ m_0}(z)\in V_r$. Finally
	$$
	w=h_\varphi(z)+m_0-(n_0+m_0)=h_\varphi(\varphi^{\circ m_0}(z))-(n_0+m_0)\in \bigcup_{n=0}^\infty (h_\varphi(V_r)-n).
	$$
	
	Now, note that
	$
	h_\varphi\circ f\circ\mtilde \varphi=h_\varphi\circ\varphi\circ f=h_\varphi\circ f +1.
	$ Recall also that ${f(\D)=V_r}$. In combination with~\eqref{EQ_Pi-absorbing}, this means that $\Mdd:={(\Pi_r, h_\varphi\circ f,z\mapsto z+1)}$ is a semimodel of $\mtilde \varphi$.

Furthermore, by Theorem~\ref{Thm:V-sets}\+\ref{IT_V-sets(4)}, $h_\varphi$ is univalent in~$V_r$. Hence, ${h_\varphi\circ f}$ is univalent in~$\UD$ and as a result, the semimodel~$\Mdd$ is actually a holomorphic model for~$\mtilde\varphi$. By Theorem~\ref{Thm:model}\+\ref{IT_HM-para-PHS}, it follows that~$\mtilde\varphi$ is parabolic of positive hyperbolic step, as desired.
\end{proof}

We will make use of the following corollary of the above proposition.
Set $V:=V_{\varphi}(1/3)$ and let $f$ be a conformal map of~$\UD$ onto~$V$. Then ${\varphi(V)\subset V}$ and ${\psi(V)\subset V}$ for any~${\psi\in\Zn(\varphi)}$, see Theorem~\ref{Thm:V-sets}\+\ref{IT_V-sets(1)} and Lemma~\ref{Le:V-sets}\+\ref{V-sets-fact1}. Consider the map
\begin{equation}\label{Definition_B_varphi}
 \mathfrak P_\varphi:\Zn(\varphi)\to\Zn(\mtilde\varphi);~\psi\mapsto\mtilde\psi:={f^{-1}\circ \psi\circ f}
\end{equation}
that transforms a self-map $\psi\in\Hol(\UD)$ commuting with~$\varphi$ into the self-map $\mtilde \psi\in \Hol(\D)$ commuting with $\mtilde \varphi:={f^{-1}\circ \varphi\circ f\in \Hol(\D)}$.

\begin{corollary} \label{Cor:projection} The map $\mathfrak P_\varphi:\Zn(\varphi)\to\Zn(\mtilde\varphi)$ defined in \eqref{Definition_B_varphi} is a homeomorphism onto its image~$\mathfrak P_\varphi(\Zn(\varphi))$, which is a relatively closed subset of~$\Zn(\mtilde\varphi)$.
\end{corollary}
\begin{proof}
 Firstly, $\mathfrak P_\varphi$ is injective as a consequence of the identity principle for holomorphic functions. Further, consider a sequence ${(\psi_n)\subset\Zn(\varphi)}$ and denote $\mtilde\psi_n:=\mathfrak P_\varphi(\psi_n)$ for all ${n\in\Natural}$.
 If $(\psi_n)$ converges to some ${\psi\in\Zn(\varphi)}$ then, by Lemma \ref{Le:V-sets}\+\ref{V-sets-fact1},  $\psi(V)\subset V$ and hence
$$
  \mtilde\psi_n =f^{-1}\circ\psi_n\circ f~\to~f^{-1}\circ\psi\circ f=\mathfrak P_\varphi(\psi)\quad\text{as~$~n\to+\infty$},
$$
i.e. the map $\mathfrak P_\varphi$ is continuous. Conversely, suppose that $(\mtilde\psi_n)$ converges to some ${\mtilde\psi\in\Zn(\mtilde\varphi)}$. Passing to the limit in ${\psi_n\circ f=f\circ\mtilde\psi_n}$, we see that the limit~$\psi$ of any convergent subsequence of~$(\psi_n)$ satisfies
\begin{equation}\label{EQ_Gum-EQ}
\psi\circ f=f\circ\mtilde\psi.
\end{equation}
By a standard argument based on normality of the family~$\Hol(\UD)$, the latter implies in turn that $(\psi_n)$ converges to some $\psi$ with $\psi(V)=f(\mtilde\psi(\UD))\subset f(\UD)=V$. In particular, ${\psi\in\Hol(\UD)}$. Therefore, we may pass to the limit in ${\psi_n\circ\varphi}={\varphi\circ\psi_n}$ and conclude that ${\psi\in\Zn(\varphi)}$. Appealing again to~\eqref{EQ_Gum-EQ}, we get $\mtilde\psi=\mathfrak P_\varphi(\psi)$. This shows that $\mathcal Q:=\mathfrak P_\varphi(\Zn(\varphi))$ is a relatively closed in~$\Zn(\mtilde\varphi)$ and that $\mathfrak P_\varphi^{-1}$ is continuous in~$\mathcal Q$.
\end{proof}

Finally, we recall the following uniqueness result for the Koenigs function of parabolic self-maps of zero hyperbolic step.

\begin{theorem} [\protect{\cite[Theorem 3.1]{CDP-Abel}}]\label{Thm:uniqueness}
Let $\varphi \in \Hol(\D)$ be a parabolic self-map of
zero hyperbolic step with Koenigs function  ${h_{\varphi} :\mathbb{D}\rightarrow \mathbb{C}}$.
Let $z_{0}\in \mathbb{D}$ such that ${h'_{\varphi}(z_{0})\neq 0}$. Further, let
${h_{1}: \mathbb{D}\rightarrow \mathbb{C}}$ be a holomorphic function satisfying the Abel
functional equation ${h_{1} \circ \varphi =h_{1} +1}$. Then the following conditions are equivalent:
\begin{romlist}
 \item $h_{1}=h_{\varphi}+c$ for some constant~${c\in\C}$;
 \item  there exist ${r>0}$ and ${N\in \mathbb{N}}$ such that $h_{1}$ is
univalent on every hyperbolic disc of center $\varphi^{\circ n}(z_{0})$ and
radius ${r}$ for all $n>N.$
\end{romlist}
\end{theorem}

\section{Simultaneous linearization and pseudo-iteration semigroup}
\label{S_general_facts}
In~\cite{Cowen-comm} Cowen introduced  the notion of pseudo-iteration semigroup \label{pseudo-iteration} playing the central role in his study of commuting holomorphic self-maps of~$\UD$.  For a non-elliptic self-map  ${\varphi\in\Hol(\UD)}$, this notion can be defined as follows. Using conformal mappings one can pass from the canonical holomorphic model for~$\varphi$ to another holomorphic model ${(S,f,\alpha)}$ for~$\varphi$ such that the base space~$S$ is either~$\UD$ or~$\C$. A self-map ${\psi\in \Hol(\D)}$ is said to belongs to the {\sl pseudo-iteration semigroup} of~$\varphi$ if there exists a M\"obius transformation~$\beta$ of the Riemann sphere~$\overline\C$ such that
$$
\alpha\circ\beta=\beta\circ\alpha \quad\text{and}\quad f\circ \psi=\beta \circ f.
$$
The second equality implies, of course, that $\beta(S)\subset S$. Furthermore, using the fact that commuting M\"obius transformations  have the same fixed points (except for the case when both are involutions) see e.g. \cite[Corollary~1.6.21]{Abate2}, and passing back to the canonical holomorphic model, it is not difficult to see that ${\psi\in \Hol(\D)}$ belongs to the  pseudo-iteration semigroup of a non-elliptic self-map ${\varphi\in\Hol(\UD)}$ if and only if
$$
  h_\varphi\circ \psi=h_\varphi+c \quad\text{for some constant~$~c\in\C$.}
$$
Since the Koenigs function $h_\varphi$ of~$\varphi$, by the very definition, satisfies Abel's equation ${h_\varphi\circ \varphi=h_\varphi+1}$,  this explains the relation of the pseudo-iteration semigroup to the notion of simultaneous linearization, see Definition~\ref{DF_SimLin} below. As we have seen in~\cite{CDG-Centralizer} for univalent self-maps and as it will become evident in~\cite{CDG-positive-h.step}, the simultaneous linearization turns out to be a more convenient tool than the general concept of pseudo-iteration when we restrict the study of commuting self-maps to the (most interesting and complicated) non-elliptic case.

\begin{definition}\label{DF_SimLin}
 Let $\varphi,\psi\in\Hol(\UD)$ be two self-maps at least one of which is non-elliptic. We say that $\varphi$ and $\psi$ admit \textsl{simultaneous linearization} if there exist constants ${c_1,c_2\in\Complex}$ with ${|c_1|^2+|c_2|^2\neq0}$ and a holomorphic function ${h:\UD\to\C}$ satisfying
 \begin{equation}\label{EQ_SimLinSYS}
   h\circ \varphi=h+c_{1}, \qquad h\circ \psi=h+c_{2}.
 \end{equation}
\end{definition}

\begin{remark}
It is worth mentioning that simultaneous linearization, as a general concept, have been studied in connection with commutativity in other contexts, e.g. for torus diffeomorphisms and holomorphic germs; see \cite{SL,SLJ1,SLJ2} and references therein.
\end{remark}

The main result of this preparatory section, Proposition~\ref{PR_affine-part} below, shows that the simultaneous linearization, with a slight additional assumption imposed on the function~$h$, provides a sufficient condition for commutativity.

\begin{proposition}\label{PR_affine-part}
Let $\varphi, \psi \in\Hol(\UD)$ be two parabolic self-maps having the same Denjoy\,--\,Wolff point~$\tau$.
If there exists $c_{1}, c_{2}\in \C$ and a holomorphic function $h:\D\to \C$ which is univalent in some domain $V\subset \D$ isogonal at $\tau$ and such that
\begin{equation*}
h\circ \varphi=h+c_{1}, \qquad h\circ \psi=h+c_{2}.
\end{equation*}
Then $\varphi \circ \psi=\psi\circ \varphi$.\smallskip

In particular, if  ${h_\varphi\circ\psi}={h_\varphi+c}$ for some constant~${c\in\Complex}$, then ${\psi\in\Zn(\varphi)}$.
\end{proposition}
The hypothesis that $\varphi$ and $\psi$ have the same DW-point~$\tau$ cannot be eliminated from the above proposition, see Example~\ref{EX_DW-importante} in the next section.

A related result can be found in \cite[Proposition~2.2]{Cowen-comm}. However, a principal difference is that Proposition~\ref{PR_affine-part} gives a sufficient condition for $\varphi$ and $\psi$ to commute \textit{without} assuming that $\psi$ is a pseudo-iteration of~$\varphi$. Accordingly, the simultaneous linearization \textit{a priori} does not have to be given by the Koenigs function~$h_\varphi$ of~$\varphi$. It does so (and the condition becomes also necessary) when $\varphi$ is of zero hyperbolic step; see Theorem~\ref{Thm:zero-hyperbolic-step}\+\ref{IT_THzero-SimLin} in the next section. In contrast, for the case of positive hyperbolic step it will be shown in~\cite{CDG-positive-h.step}, see also~\cite{CDG-Centralizer}, that $h$ in~\eqref{EQ_SimLinSYS} can be essentially different from~$h_\varphi$.

For the proof of Proposition~\ref{PR_affine-part} we need the following lemma, which we regard as known to specialists and which can be easily deduced, e.g., from  \cite[Lemma~A.1 and Proposition~A.6]{GuKouMouRoth}.
\begin{lemma}\label{LM_isogonality}
Let $f\in\Hol(\UD)$ and~$\sigma\in\UC$. Suppose that the angular limit~$f(\sigma):=\anglim_{z\to\sigma}f(z)$ exists and belongs to~$\UC$ and suppose that the angular derivative
$$
 f'(\sigma):=\anglim_{z\to\sigma}\frac{f(z)-f(\sigma)}{z-\sigma}
$$
is finite. Then the preimage $f^{-1}(V)$ of any domain ${V\subset\UD}$ isogonal at~$f(\sigma)$ contains a domain ${U\subset\UD}$ isogonal at~$\sigma$.
\end{lemma}
For a point~$\tau\in\overline\UD$, denote by $\Hol_\tau(\UD)$ the family of self-maps consisting of~$\id_\UD$ and all~${\psi\in\Hol(\UD)\setminus\{\id_\UD\}}$ whose DW-point coincides with~$\tau$.
\begin{proof}[\proofof{Proposition~\ref{PR_affine-part}}] Denote $f_1:={\varphi\circ\psi}$ and $f_2:={\psi\circ\varphi}$. Both self-maps belong to $\Hol_{\tau}(\D)$. In particular, both satisfy the hypothesis of Lemma \ref{LM_isogonality} with $\sigma=\tau$.
Observe first of all that it follows easily from the hypothesis  that
\begin{equation}\label{EQ_h-equality}
 h\circ f_{1}= h\circ(\varphi\circ\psi)=h_\varphi+c_{1}+c_{2}=h\circ(\psi\circ\varphi)=h\circ f_{2}.
\end{equation}
Since the domain~$V$ is isogonal at~$\tau$, applying Lemma~\ref{LM_isogonality} to $f_1$ and~$f_2$ with ${\sigma:=\tau}$, we see that there are two domains ${U_1,U_2\subset\UD}$ both isogonal at~$\tau$ and such that ${f_k(U_k)\subset V}$,  ${k=1,2}$.
It follows easily from the very definition that the intersection of any finite number of domains isogonal at the same point of~$\UC$ has non-empty interior. Set ${B:=U_1\cap U_2}$. Therefore, ${f_k(B)\subset V}$,  ${k=1,2}$, and by the univalence of $h$ in~$V$ and the identity principle for holomorphic functions, we obtain that ${f_{1}=f_{2}}$ on the unit disc. This proves the first assertion of the proposition.

Concerning the second assertion, it is enough to recall that, by Theorem~\ref{Thm:V-sets}, $h_\varphi$ is univalent in the domain ${V:=V_\varphi(1/3)}$, which is isogonal at $\tau$.
\end{proof}

We conclude this section by deducing some further conclusions assuming that the hypothesis of the second assertion in Proposition~\ref{PR_affine-part} holds with some~${c\in\Real}$.

\begin{proposition} \label{Prop:firstpropertiescoefficient}
Let $\varphi\in\Hol(\D)$ be parabolic with Koenigs function~$h_\varphi$ and Denjoy\,--\,Wolff point ${\tau\in\partial\UD}$.  Suppose ${\psi\in\Hol_\tau(\UD)}$ satisfies ${h_\varphi\circ\psi}={h_\varphi+c}$ for some constant~${c\in\R}$. The following statements hold.
\begin{ourlist}
		\item\label{firstpropereties1}  $c=0$ if and only if $\psi=\id_\D$.
		\item\label{firstpropereties2} If $c=\frac{m}{n}\in\Q$ with $m,n\in\N$, then $\psi^{\circ n}=\varphi^{\circ m}$.
		\item\label{firstpropereties3}  If $c<0$, then $\varphi$ as well as $\psi$ are parabolic automorphisms. Moreover, if $c=-\frac{m}{n}\in\Q$ with $m,n\in\N$, then $\psi^{\circ n}=\varphi^{\circ -m}$.	
	\end{ourlist}		

\end{proposition}
\begin{proof}
Let us denote $V:=V_\varphi(1/3)$. By Theorem~\ref{Thm:V-sets}, $h_{\varphi}$ and $\varphi$ are univalent in $V$ and ${\varphi(V)\subset V}$.
Since $\psi$ commutes with~$\varphi$ by Proposition~\ref{PR_affine-part}, using Lemma~\ref{Le:V-sets}\+\ref{V-sets-fact1}, we see that ${\psi(V)\subset V}$. Moreover, $\psi$ is univalent in~$V$. Indeed, if $z_{1}, z_{2}\in V$ are such that $\psi(z_{1})=\psi(z_{2})$, then
$$
h_{\varphi}(z_{1})+c=h_{\varphi}(\psi(z_{1}))=h_{\varphi}(\psi(z_{2}))=h_{\varphi}(z_{2})+c.
$$
So that $z_{1}=z_{2}$ by the univalence of $h_{\varphi}$ in $V$.

\StepP{\ref{firstpropereties1}} If $c=0$,  then ${h_\varphi\circ\psi}={h_\varphi}$ and  $\psi|_{V}=\id_V$ because  ${\psi(V)\subset V}$ and $h_{\varphi}$ is univalent in~$V$. By the identity principle, ${\psi=\id_\D}$. The converse implication is trivial.

\StepP{\ref{firstpropereties2}}  If $c=\frac{m}{n}$ with some $m,n\in \N$, then
	$$
	h_\varphi\circ \psi^{\circ n}=h_\varphi+m=h_\varphi\circ\varphi^{\circ m}.
	$$
	Bearing in mind that $\psi(V)\subset V$ and $\varphi(V)\subset V$, we have that $\psi^{\circ n}|_V=\varphi^{\circ m}|_V$ and, by the identity principle, $\psi^{\circ n}=\varphi^{\circ m}$.
	
\StepP{\ref{firstpropereties3}}
Take $f:\D\to V$, a Riemann map from $\D$ onto  $V$. Then $\mtilde \varphi:={f^{-1}\circ \varphi\circ f}$ and $\mtilde \psi:={f^{-1}\circ \psi\circ f}$ are univalent holomorphic self-maps of~$\UD$ and moreover,
\begin{equation}\label{EQ_new}
h_{\varphi}\circ f\circ \mtilde \varphi=h_{\varphi}\circ f+1\quad\text{and}\quad
h_{\varphi}\circ f\circ \mtilde \psi=h_{\varphi}\circ f+c.
\end{equation}
By Proposition~\ref{modelV_par}, ${h:=h_{\varphi}\circ f}$ coincides up to an additive constant with the Koenigs function of~$\mtilde \varphi$. Recall also that $h_\varphi$ is univalent in~$V$ by Theorem~\ref{Thm:V-sets}\+\ref{IT_V-sets(4)}. As a consequence, $h$ is univalent in~$\UD$. Furthermore, from the second equality in~\eqref{EQ_new} it follows that
\begin{equation}\label{EQ_Omega_plus_c}
{h(\UD)+c}={h\circ\mtilde\psi(\UD)}\subset h(\UD).
\end{equation}

If $c=-m/n$ with ${m,n\in\Natural}$, then we have
 $$
   h\circ\mtilde\psi^{\circ n}\circ\mtilde\varphi^{\circ m}=h\circ\mtilde\varphi^{\circ m}+nc
     =h+m+nc=h.
 $$
Hence, $f^{-1}\circ\psi^{\circ n}\circ\varphi^{\circ m}\circ f=\mtilde\psi^{\circ n}\circ\mtilde\varphi^{\circ m}=\id_\UD$. With the help of the identity principle we may therefore conclude that ${\psi^{\circ n}\circ\varphi^{\circ m}}=\id_\UD$. Thus, $\varphi,\psi\in\Aut$ and $\psi^{\circ n}=\varphi^{\circ-m}$, as desired.

It remains to see that $\varphi,\psi\in\Aut$ also in case ${c\in(-\infty,0)\setminus\mathbb Q}$.
Recalling~\eqref{EQ_Omega_plus_c} and using \cite[Theorem~3.1\+(B)]{CDG-Centralizer}, we see that ${\mtilde\varphi\in\Aut}$.
Hence, $\varphi$ maps~$V$ conformally onto itself. We wish to show that ${\varphi\in\Aut}$. This is immediate if~${V=\UD}$, so we suppose that ${V\neq\UD}$. Fix some ${z_1\in\UD\cap\partial V}$. Then ${z_2:=\varphi(z_1)}$ also lies on~$\partial V$. Taking into account the definition of~$V$, we have ${\phd_{\UD}(z,\varphi(z))=1/3}$ for any ${z\in\UD\cap\partial V}$. In particular,
$$
  \phd_{\UD}\big(z_1,z_2\big)=\phd_{\UD}\big(z_1,\varphi(z_1)\big)=1/3\quad \text{and}\quad \phd_{\UD}\big(\varphi(z_1),\varphi(z_2)\big)=\phd_{\UD}\big(z_2,\varphi(z_2)\big)=1/3.
$$
Therefore, $\varphi\in\Aut$ by the Schwarz\,--\,Pick Lemma. It follows easily that $\psi$ is also an automorphism because $\psi={h_\varphi^{-1}\circ(h_\varphi+c)}$, where the Koenigs function $h_\varphi$ of~$\varphi$ is a conformal mapping of~$\UD$ onto a half-plane, with $\partial h_\varphi(\UD)$ being parallel to~$\Real$.
\end{proof}

\section{Characterization of commutativity via simultaneous linearization}\label{S_zero}
In his seminal paper~\cite{Cowen-comm}, Cowen proved that if two self-maps ${\varphi,\psi\in\Hol(\UD)}$ commute, then they belong to the pseudo-iteration semigroup of~${\varphi\circ\psi}$. Moreover, if additionally $\varphi$ is elliptic or hyperbolic then as he showed, $\psi$ belongs also to the pseudo-iteration semigroup of~$\varphi$. Conversely, if $\psi$ belongs to the pseudo-iteration semigroup of an elliptic or hyperbolic self-map ${\varphi\in\Hol(\UD)}$ and if it has the same Denjoy\,--\,Wolff point as~$\varphi$, then $\psi$ and~$\varphi$ commute, see e.g. Gentili and Vlacci \cite[Theorem~2.7]{Gentili-Vlacci} and Bisi and Gentili \cite[Theorem~6]{BisGen01}; see also Vlacci~\cite{Vlacci}.
Taking into account the relationship of the pseudo-iteration semigroup of a non-elliptic self-map to simultaneous linearization, discussed in the previous section, the following theorem can be seen as an analogue of these results for parabolic self-maps of zero hyperbolic step.

\begin{theorem}\label{Prop:pseuso-semigroup}
Let $\varphi\in \Hol(\D)$ be parabolic of zero hyperbolic step and let $\psi\in \Hol(\D)$ be different from $\id_\D$. Likewise, let $h_{\varphi}$ be the Koenigs function of $\varphi$. Then, the following are equivalent:
\begin{romlist}
\item\label{IT-EQUIV_commute} $\psi\in \Zn(\varphi)$;
\item\label{IT-EQUIV_SimLin} $\varphi$ and $\psi $ have the same Denjoy\,--\,Wolff point and
$\,h_\varphi\circ \psi=h_\varphi+c\,$ for some constant $c\in\C$.
\end{romlist}
\end{theorem}

\noindent Implication \ref{IT-EQUIV_SimLin}~$\Rightarrow$~\ref{IT-EQUIV_commute} in this theorem has been already proved in the previous section: it is exactly the second statement in Proposition~\ref{PR_affine-part}. The other implication \ref{IT-EQUIV_commute}~$\Rightarrow$~\ref{IT-EQUIV_SimLin} follows directly from Behan's Theorem and assertion~\ref{IT_THzero-SimLin} of the more technical Theorem~\ref{Thm:zero-hyperbolic-step}, which we will prove below.

\begin{remark}\label{RM_under-additional-condition}
For parabolic self-maps, \cite[Theorem~6]{BisGen01} implies the equivalence \ref{IT-EQUIV_commute}~$\Leftrightarrow$~\ref{IT-EQUIV_SimLin} under the assumption that at least one (and hence every) orbit w.r.t.~$\varphi$ converges to the Denjoy\,--\,Wolff point non-tangentially. In fact, in \cite[Remark on page~40]{Gentili-Vlacci}
it is asked whether this additional hypothesis can be removed.
Theorem~\ref{Prop:pseuso-semigroup} answers positively this question for parabolic self-maps of zero hyperbolic step\footnote{It is worth recalling
that every parabolic self-map having an orbit that converges to the Denjoy\,--\,Wolff
point non-tangentially is necessarily of zero hyperbolic step \cite[page~440]{Pom79}. At the same time (see e.g. \cite[\S3.2]{ParaZoo})
there exist parabolic self-maps of~$\UD$ all of whose orbits converge to the Denjoy\,--\,Wolff point in the
tangential way. For further details, see \cite[\S4.6]{Abate2}.}.  At the same time, as we will see in \cite{CDG-positive-h.step}, and as it has been shown for univalent self-maps in~\cite[Section~7]{CDG-Centralizer},
 the answer in the general case is negative, because the implication \ref{IT-EQUIV_commute}~$\Rightarrow$~\ref{IT-EQUIV_SimLin} does not hold for parabolic self-maps of positive hyperbolic step.
\end{remark}

\begin{definition} \label{simul-linea-coefficient-zero}
Let $\varphi \in \Hol(\D)$ be a parabolic self-map of zero hyperbolic step  and let $\psi\in\Zn(\varphi)$. We will denote by~$c_{\varphi,\psi}$ the constant ${c\in\C}$ defined in a unique way by the relation ${h_\varphi\circ\psi=h_\varphi+c}$ in Theorem~\ref{Prop:pseuso-semigroup}\+\ref{IT-EQUIV_SimLin}. This constant will be called the {\sl simultaneous linearization coefficient} of~$\psi$ w.r.t.~$\varphi$.
\end{definition}

Note that according to Proposition~\ref{Prop:firstpropertiescoefficient}, $c_{\varphi,\psi}=0$ if and only if $\psi=\id_{\D}$.
At the end of this section we will establish basic properties of the map ${\Zn(\varphi)\ni\psi\mapsto c_{\varphi,\psi}\in\C}$, see Theorem~\ref{Thm:homeomorphism}. As a corollary, we will prove the following remarkable result. (For the definition of a one-parameter semigroup, see Appendix, page~\pageref{PAGE_one-paremeter-semigroup}.)

\begin{theorem}\label{Thm:embedding}
Let $\varphi \in \Hol(\D)$ be parabolic of zero hyperbolic step. If $\id_{\D}$ is not isolated in $\Zn(\varphi)$, then $\varphi$ is univalent and $  \Zn(\varphi)$ contains a non-trivial continuous one-parameter semigroup~$(\phi_t)$. If in addition, $\varphi$ has a boundary fixed point other than its Denjoy\,--\,Wolff point, then $\varphi=\phi_{t_0}$ for some~${t_0>0}$ (and hence, by the very definition,  $\varphi$ is embeddable).
\end{theorem}

The next theorem gathers the rest of the results we are going to establish in this section.

\begin{theorem}\label{Thm:zero-hyperbolic-step}
Let $\varphi:\D\to \D$ be parabolic of zero hyperbolic step and $\psi \in \Zn(\varphi)$. Let $h_\varphi$ be the Koenigs map of $\varphi$.
Then the following four statements hold.
\begin{ourlist}
\item\label{IT_THzero-univalent}
          If $\varphi $ is univalent, then so is $\psi$.
\item\label{IT_THzero-univalent2}
          Let $V\subset\UD$ be a simply connected  and $\varphi$-absorbing domain for $\varphi$.  If ${\psi(V)\subset V}$ and $\varphi$ is univalent in~$V$, then $\psi$ is univalent in~$V$ as well.
\item\label{IT_THzero-SimLin}
          There exits $c\in \C$ such that $h_\varphi\circ \psi=h_\varphi+c$.
\item\label{IT_THzero-abelian}
          $\Zn(\varphi)$ is abelian, i.e., if $\psi_{j}\in \Zn(\varphi)$, $j=1,2$, then $\psi_{1}\circ\psi_{2}=\psi_{2}\circ \psi_{1}$.
\end{ourlist}
\end{theorem}

\begin{remark}
Assertions~\ref{IT_THzero-univalent} and~\ref{IT_THzero-abelian} of the above theorem extend Cowen's results \cite[Corollaries~4.2 and~4.9]{Cowen-comm} to parabolic self-maps of zero hyperbolic step. As we will see in~\cite{CDG-positive-h.step}, these statements fail in the case of positive hyperbolic step.
\end{remark}

Before passing to the proof of the results stated above, it
is also worth pointing out that for \textit{univalent} parabolic  self-maps~$\varphi$ of zero hyperbolic step, by \cite[Proposition~4.3]{CDG-Centralizer} we have
\begin{fleqn}[\parindent]
\begin{equation*}
\qquad
 \Zn(\varphi)=\{\psi \in \Hol(\D):\, \textrm{ there exists } c\in \C  \textrm{ such that } h_\varphi\circ \psi=h_\varphi+c\}.
\end{equation*}
At the same time by Theorem~\ref{Prop:pseuso-semigroup},
\begin{multline*}
\qquad
 \Zn(\varphi)=\{\psi \in \Hol_\tau(\D):\, \textrm{ there exists } c\in \C  \textrm{ such that } h_\varphi\circ \psi=h_\varphi+c\}\\
  ~\subset~
\{\psi \in \Hol(\D):\, \textrm{ there exists } c\in \C  \textrm{ such that } h_\varphi\circ \psi=h_\varphi+c\},
\end{multline*}
\end{fleqn}
where $\tau$ is the Denjoy\,--\,Wolff point of~$\varphi$ and we recall that $\Hol_\tau(\UD)$ stands the family of self-maps consisting of~$\id_\UD$ and all~${\psi\in\Hol(\UD)\setminus\{\id_\UD\}}$ whose Denjoy\,--\,Wolff point coincides with~$\tau$.
As the example below shows, if~$\varphi$ is not univalent, then the inclusion is in general strict.

\begin{example}\label{EX_DW-importante}
Take $\varphi$ a non-elliptic self-map of~$\UD$ such that $\varphi(-z)=\varphi(z)$ for all $z\in \D$ and let $h_{\varphi}$ be its Koenigs function. (One example of such a self-map is $\varphi(z):={(z^2+a)/(a+1)}$ with ${a\ge1}$.) Then
$$
h_{\varphi}(-z)=h_{\varphi} (\varphi(-z))-1=h_{\varphi} (\varphi(z))-1=h_{\varphi}(z), \quad z\in \D.
$$ The function $\psi:=-\varphi$ does not commute with  $\varphi$ but the above property shows that $h_{\varphi}\circ \psi=h_{\varphi}+1$.
\end{example}

\begin{proof}[\proofof{Theorem~\ref{Thm:zero-hyperbolic-step}}]
If $\psi=\id_\UD$, then statements~\ref{IT_THzero-univalent}\,--\,\ref{IT_THzero-SimLin} hold trivially. So we may suppose that ${\psi\neq\id_\UD}$. Then by Cowen's result, see Section~\ref{SS_commuting}, $\psi$ is parabolic. This allows us to take advantage of Lemma~\ref{Le:V-sets} and use Theorem~\ref{Thm:V-sets} both for $\varphi$ and~$\psi$.

Now write $h_{1}:=h_{\varphi}\circ \psi$. Then ${h_1\circ\varphi}={h_{\varphi}\circ\psi\circ\varphi} ={h_{\varphi}\circ\varphi\circ\psi}={(h_{\varphi}+1)\circ\psi}={h_1+1}$, i.e. $h_1$ satisfies Abel's equation for~$\varphi$.

\StepP{\ref{IT_THzero-univalent}} The function $h_\varphi$ is univalent because $\varphi$ is univalent (see \cite[Lemma~3.5.4]{Abate2}).
Take ${z_{0}\in V_{\psi}(1/9)}$ and write $z_{n}:=\varphi^{\circ n}(z_{0})$. By Lemma~\ref{Le:V-sets}\+\ref{V-sets-fact4}, ${\phdisk(z_{n},1/9)\subset V_{\psi}(1/3)}$ for all~${n\in\Natural}$. By Theorem~\ref{Thm:V-sets}\+\ref{IT_V-sets(4)}, $\psi$ is univalent in $V_{\psi}(1/3)$. As a consequence,  $h_{1}={h_{\varphi}\circ \psi}$ is univalent in $\phdisk(z_{n},1/9)$ for any~${n\in\Natural}$. Therefore, by Theorem~\ref{Thm:uniqueness}, there is a constant~$c$ such that ${h_{\varphi}\circ \psi}=h_{1}={h_{\varphi}+c}$. Thus, $\psi={h_{\varphi}^{-1}\circ(h_{\varphi}+c)}$ is univalent.

\StepP{\ref{IT_THzero-univalent2}} Take $f:\D\to V$ a Riemann map of $V$ and define $\mtilde \varphi:=f^{-1}\circ \varphi\circ f$. Clearly, $\mtilde \varphi$ is univalent and, by Proposition~\ref{modelV_par}, it is of zero hyperbolic step.
Define ${\mtilde \psi:=f^{-1}\circ \psi\circ f}$. Since $\psi(V)\subset V$, $\mtilde \psi$ is a well-defined holomorphic self-map of~$\UD$ commuting with $\mtilde \varphi$. Thus, by~\ref{IT_THzero-univalent}, $\mtilde \psi$ is univalent and then so is $\psi$ in~$V$.

\StepP{\ref{IT_THzero-SimLin}} In part, our argument here looks similar to that we have used to prove~\ref{IT_THzero-univalent}, but it is important to emphasize that now we work with $V_\varphi(1/9)$ and $V_\varphi(1/3)$ instead of the analogous sets defined for~$\psi$.  Take ${z_{0}\in V_{\varphi}(1/9)}$ and write ${z_{n}:=\varphi^{\circ n}(z_{0})}$. By Lemma~\ref{Le:V-sets}\+\ref{V-sets-fact4a} applied for~$\varphi$ in place of~$\psi$, we have ${\phdisk(z_{n},1/9)\subset V_{\varphi}(1/3)}$.
Recall that $V_{\varphi}(1/3)$ is a $\varphi$-absorbing set and that $\varphi$ and $h_{\varphi}$ are univalent in~$V_{\varphi}(1/3)$, thanks to Theorem~\ref{Thm:V-sets}. Moreover, by Lemma~\ref{Le:V-sets}\+\ref{V-sets-fact1},
${\psi (V_{\varphi}(1/3))\subset V_{\varphi}(1/3)}$. As a consequence, $\psi$ is univalent in~$V_{\varphi}(1/3)$ by statement~\ref{IT_THzero-univalent2} for~$V:=V_\varphi(1/3)$. Thus, $h_{1}:={h_{\varphi}\circ \psi}$ is univalent in $V_{\varphi}(1/3)$ and hence in $\phdisk(z_{n},1/9)$ for all~${n\in\Natural}$. Applying again Theorem \ref{Thm:uniqueness}, we see that there is a constant $c$ such that $h_{\varphi }\circ \psi=h_{1}=h_{\varphi}+c$, as desired.

\StepP{\ref{IT_THzero-abelian}} Let ${\psi_{1},\psi_{2}\in\Zn(\varphi)}$.
Take ${f:\D\to V_{\varphi}(1/3)}$  a Riemann map of $V_{\varphi}(1/3)$ and define $\mtilde \varphi:={f^{-1}\circ \varphi\circ f}$, $\mtilde \psi_{1}:={f^{-1}\circ \psi_{1}\circ f}$ and $\mtilde \psi_{2}:={f^{-1}\circ \psi_{2}\circ f}$.
Then by Proposition~\ref{modelV_par}\+\ref{IT_modelV-zero-step}, $\mtilde \varphi$ is parabolic of zero hyperbolic step. Furthermore, in view of~Lemma~\ref{Le:V-sets}\+\ref{V-sets-fact1}, $\mtilde\psi_1$ and $\mtilde\psi_2$ are well-defined  self-maps of~$\UD$ that commute with~$\mtilde \varphi$.
With the help of~\ref{IT_THzero-SimLin},  it follows easily that
$$h_\varphi\circ f\circ \mtilde\psi_1\circ\mtilde\psi_2=h_\varphi\circ f\circ \mtilde\psi_2\circ\mtilde\psi_1.
$$
Since $h_\varphi\circ f\in\U(\D,\C)$ by Theorem~\ref{Thm:V-sets}\+\ref{IT_V-sets(4)}, we have $$\mtilde\psi_1\circ\mtilde\psi_2=\mtilde\psi_2\circ\mtilde\psi_1.
$$
Thus $\psi_{1}|_{V_{\varphi}(1/3)}$ and $\psi_{2}|_{V_{\varphi}(1/3)}$ commute and so do $\psi_{1}$ and~$\psi_{2}$.
\end{proof}

Basic properties of the map that takes $\psi\in\Zn(\varphi)$ to the simultaneous linearization coefficient of~$\psi$ w.r.t.~$\varphi$ are given in the following result, which is an analogue of Pragner's Theorem for centralizers of elliptic self-maps~\cite[Theorem~3]{Pranger}; compare with~\cite[Theorem~2]{BisGen01} and with our earlier results~\cite[Theorems~3.5 and~5.2]{CDG-Centralizer} for the univalent case.

\begin{theorem}\label{Thm:homeomorphism}
Let $\varphi\in\Hol(\UD)$ be parabolic of zero hyperbolic step. Then the map
$$
  \Zn(\varphi)\ni\psi \mapsto c_{\varphi,\psi}\in\C
$$
is a homeomorphism onto a closed subset $\mathcal A_\varphi\subset\C$. Moreover,
\begin{equation}
 c_{\varphi,\psi_1\circ\psi_2}=c_{\varphi,\psi_1}+c_{\varphi,\psi_2}\quad\text{for any~$~\psi_1,\psi_2\in\Zn(\varphi)$}
\end{equation}
and, as a consequence, $[\mathcal A_\varphi,+]$ is an additive semigroup containing~$\N_0$.
\end{theorem}

\begin{proof}
Set $V:=V_{\varphi}(1/3)$ and let $f$ be a conformal map of~$\UD$ onto~$V$. Then ${\varphi(V)\subset V}$ and ${\psi(V)\subset V}$ for any~${\psi\in\Zn(\varphi)}$, see Theorem~\ref{Thm:V-sets}\+\ref{IT_V-sets(1)} and Lemma~\ref{Le:V-sets}\+\ref{V-sets-fact1}. Consider the map
$$
 \mathfrak P_\varphi:\Zn(\varphi)\to\Zn(\mtilde\varphi);~\psi\mapsto\mtilde\psi:={f^{-1}\circ \psi\circ f}
$$
that transforms a self-map $\psi\in\Hol(\UD)$ commuting with~$\varphi$ into the self-map $\mtilde \psi\in \Hol(\D)$ commuting with $\mtilde \varphi:={f^{-1}\circ \varphi\circ f\in \Hol(\D)}$.

By Theorem~\ref{Thm:V-sets}\+\ref{IT_V-sets(4)}, $\mtilde\varphi$ is univalent, and by Proposition~\ref{modelV_par}\+\ref{IT_modelV-zero-step}, it is parabolic of zero hyperbolic step.
Therefore, by Theorem~\ref{Thm:zero-hyperbolic-step}\+\ref{IT_THzero-univalent}, all elements of~$\Zn(\mtilde\varphi)$ are univalent. These facts allow us to apply~\cite[Theorems~3.1 and~5.2]{CDG-Centralizer}, according to which the map
$$
  \mathfrak T_{\mtilde\varphi}:\mathcal A_{\mtilde\varphi}\to\Zn(\mtilde\varphi)\subset\U(\UD);~
    c\mapsto\mtilde\psi_c:=h_{\mtilde\varphi}^{-1}\circ\big(h_{\mtilde\varphi}+c\big)
$$
is a homeomorphism of the closed set $\mathcal A_{\mtilde\varphi}:=\{c\in\Complex\colon \mtilde\Omega+c\subset\mtilde\Omega\}\subset\C$, ${\mtilde\Omega:=h_{\mtilde\varphi}(\UD)}$, onto the centralizer of~$\mtilde\varphi$.

By Proposition~\ref{modelV_par}\+\ref{IT_modelV-zero-step} and Remark~\ref{RM_normalization}, the Koenigs map $h_{\mtilde\varphi}$ coincides, up to an additive constant, with~$h_\varphi\circ f$. Therefore, given any~${\psi\in\Zn(\varphi)}$, the equality ${h_\varphi\circ\psi=h_\varphi+c_{\varphi,\psi}}$ implies that ${h_{\mtilde\varphi}\circ\mtilde\psi=h_{\mtilde\varphi}+c_{\varphi,\psi}}$ with ${\mtilde\psi:=\mathfrak P_\varphi(\psi)}$. As a result, we can conclude that  ${\Zn(\varphi)\ni\psi\mapsto c_{\varphi,\psi}}$ coincides with the composition $\mathfrak T_{\mtilde\varphi}^{-1}\circ\mathfrak P_\varphi$.
By Corollary \ref{Cor:projection},  $\mathfrak P_\varphi$ is a homeomorphism onto its image~$\mathfrak P_\varphi(\Zn(\varphi))$, which is a relatively closed subset of~$\Zn(\mtilde\varphi)$.
Together with the preceeding argument, this implies that the map
$$
{\big(\Zn(\varphi)\ni\psi\mapsto c_{\varphi,\psi}\big)}=\mathfrak T_{\mtilde\varphi}^{-1}\circ\mathfrak P_\varphi
$$
is a homeomorphism of $\Zn(\varphi)$ onto the closed set
$$
 \mathcal A_\varphi:=\{c_{\varphi,\psi}:\psi\in\Zn(\varphi)\}~=~\mathfrak T_{\mtilde\varphi}^{-1}\big(\mathfrak P_\varphi(\Zn(\varphi))\big)~\subset~\mathcal A_{\mtilde\varphi}~\subset\C.
$$

It remains to observe that
$$
  c_{\varphi,\psi_1\circ\psi_2}~=~h_\varphi\circ\psi_1\circ\psi_2\,-\,h_\varphi~=~
    (h_\varphi\circ\psi_1-h_\varphi)\circ\psi_2~+~h_\varphi\circ\psi_2-h_\varphi~=~
    c_{\varphi,\psi_1}+c_{\varphi,\psi_2}
$$
for any $\psi_1,\psi_2\in\Zn(\varphi)$.
\end{proof}

\medskip
\begin{proof}[\proofof{Theorem~\ref{Thm:embedding}}]
By the hypothesis ${\Zn(\varphi)\setminus\{\id_\UD\}}$ contains a sequence $(\psi_n)$ converging locally uniformly in~$\UD$ to the identity map. We have to show that $\Zn(\varphi)$ contains a non-trivial continuous one-parameter semigroup.

By Theorem~\ref{Thm:homeomorphism}, the sequence $(c_{\varphi,\psi_{n}})$ converges to $c_{\varphi,\id_{\D}}=0$. Notice that $c_{\varphi,\psi_{n}}\neq 0$ for all~$n$.
Passing to a subsequence we may suppose that also ${\arg (c_{\varphi,\psi_{n}})\to\theta_0}$ as ${n\to+\infty}$, for some ${\theta_0\in\Real}$. Now fix some ${t>0}$ and write $k(n):=\big\lfloor t/|c_{\varphi,\psi_{n}}|\big\rfloor$, to denote the integer part of ${t/|c_{\varphi,\psi_{n}}|}$. Since
$$
k(n)|c_{\varphi,\psi_{n}}|\leq t<(k(n)+1)|c_{\varphi,\psi_{n}}| \quad \textrm{ for all } \ n,
$$
it is easy to see that
\begin{equation*}
k(n)c_{\varphi,\psi_{n}}\to t e^{i\theta_0}\quad\text{as~$~n\to+\infty$}.
\end{equation*}
Note that $k(n)c_{\varphi,\psi_{n}}$ equals the simultaneous linearization coefficient of $\psi_{n}^{\circ k(n)}\in\Zn(\varphi)$.
Since by Theorem~\ref{Thm:homeomorphism}, $\mathcal A_\varphi:=\{c_{\varphi,\psi}:\psi\in\Zn(\varphi)\}$ is a closed subset of~$\C$ containing~$0$, it follows
that ${\{t e^{i\theta_0}:t\ge0\}\subset \mathcal A_\varphi}$. This means that $\Zn(\varphi)$ contains a family $(\phi_t)_{t\ge0}$ such that
$$
 c_{\varphi,\phi_t}=t e^{i\theta_0}\quad\text{for any~$~t\ge0$}.
$$

The family $(\phi_t)_{t\ge0}$ having this property is a non-trivial continuous one-parameter semigroup because by Theorem~\ref{Thm:homeomorphism}, the map ${\psi\mapsto c_{\varphi,\psi}}$ is an isomorphism between the topological unital semigroups ${[\Zn(\varphi),\circ]}$ and ${[\mathcal A_\varphi,+]}$.
The univalence of $\varphi$ is now a consequence of \cite[Proposition~3.3]{CDG-Embedding}.

Finally, if $\varphi$ has a boundary fixed point~${\sigma\in\UC}$ different from its Denjoy\,--\,Wolff point, then $\varphi$ coincides with one of the elements of~$(\phi_t)$ by \cite[Theorem~1.4]{CDG-Embedding}.
\end{proof}

\begin{remark}
Clearly, if in the above proof $\theta_0=0$, then $\phi_1=\varphi$ and hence $\varphi$ is embeddable. At the same time, if $0$ is not in the limit set of the sequence $(\arg (c_{\varphi,\psi_{n}}))$, then $\varphi$ does not have to be embeddable, even though $\Zn(\varphi)$ contains a non-trivial continuous one-parameter semigroup  (see e.g. \cite[Example~8.4]{CDG-Centralizer}).
\end{remark}

\section{The simultaneous linearization coefficient}\label{S_SLC}
\addtocontents{toc}{\SkipTocEntry}\subsection{Statements of the results}
For a parabolic self-map $\varphi$ of zero hyperbolic step and an element~$\psi$ of its centralizer~$\Zn(\varphi)$, the simultaneous linearization coefficient $c_{\varphi,\psi}$ was introduced in the previous section with the help of the Koenigs function~$h_\varphi$ of~$\varphi$. Namely, $c_{\varphi,\psi}$ is defined as the constant ${h_\varphi\circ\psi\,-\,h_\varphi}$. Below we establish two other ways to express $c_{\varphi,\psi}$.

\begin{theorem}\label{TH_SLC(2formulas)}
Let $\varphi\in\Hol(\UD)$ be a parabolic self-map of zero hyperbolic step with Denjoy\,--\,Wolff point $\tau\in\partial\D$ and Koenigs function $h_{\varphi}$.  Let $\psi\in\Zn(\varphi)$. The following two statements hold.
\begin{statlist}
 \item\label{IT_SLC1}
   The following angular limit exists and equals the simultaneous linearization coefficient $c_{\varphi,\psi}$ of~$\psi$ \,w.\,r.\,t.~$\varphi$:
   \begin{equation}\label{EQ_SLC}
     \anglim_{z\to\tau}\frac{\psi(z)-z}{\varphi(z)-z}~=~c_{\varphi,\psi}.
   \end{equation}
 \smallskip%
  \item\label{IT_SLC2}
    Suppose that
	\begin{equation}\label{EQ_SLsystem}
	h\circ \varphi=h+c_1,\qquad h\circ \psi=h+c_2,
	\end{equation}
	for some holomorphic function $h:\UD\to\C$ and  some ${c_1,c_2\in\C}$ with $|c_1|^2+|c_2|^2\neq0$.\\ Then ${c_1\neq0}$ and ${c_2/c_1=c_{\varphi,\psi}}$.
\end{statlist}
\end{theorem}

In~\cite{CDG-Centralizer} we proved part~\ref{IT_SLC1} of the above theorem for \textit{univalent} parabolic self-maps. Here we remove the univalence assumption but suppose that $\varphi$ is of zero
hyperbolic case. The case of positive hyperbolic case will be covered in~\cite{CDG-positive-h.step}.

The second part~\ref{IT_SLC2} is a sort of uniqueness result. Indeed, by Theorem~\ref{Thm:zero-hyperbolic-step}\+\ref{IT_THzero-SimLin} and Definition~\ref{simul-linea-coefficient-zero}, the Koenigs function ${h:=h_\varphi}$ of~$\varphi$ satisfies~\eqref{EQ_SLsystem} with ${(c_1,c_2)}:={(1,c_{\varphi,\psi})}$. Taking this into account, Theorem~\ref{TH_SLC(2formulas)}\+\ref{IT_SLC2} implies that the simultaneous linearization problem~\eqref{EQ_SLsystem} admits a holomorphic solution ${h:\UD\to\C}$ if and only if ${(c_1,c_2)\in\C^2}$ is a multiple of~${(1,c_{\varphi,\psi})}$. Concerning the uniqueness of~$h$, we will prove the following statement.
\begin{proposition}\label{PR_SLC(cont)}
If under the hypothesis of Theorem~\ref{TH_SLC(2formulas)}\+\ref{IT_SLC2}, the simultaneous linearization coefficient is not a rational real number, i.e. if ${c_{\varphi,\psi}\not\in\mathbb Q},$ then
$$
 h-h(0)=c_1 h_\varphi.
$$
\end{proposition}
The hypothesis that $c_{\varphi,\psi}$ is not a rational number  in the above proposition is essential. Indeed, if ${c_{\varphi,\psi}=p/q\in\mathbb Q}$, then every function of the form $h:={h_\varphi+F\circ h_\varphi}$, where $F$ is a $1/q$-periodic entire function, satisfies~\eqref{EQ_SLsystem}.

Suppose now that the two parabolic self-maps $\varphi$ and~$\psi$ are \textit{both} of zero hyperbolic step. Then combining Theorem~\ref{Thm:zero-hyperbolic-step}\+\ref{IT_THzero-SimLin} with the above proposition applied to $\varphi$ and $\psi$ interchanged, we see that $h_\varphi=c_{\varphi,\psi} h_\psi$ provided that ${c_{\varphi,\psi}\not\in\mathbb Q}$. As our next result shows, in this case the latter condition is not essential.

\begin{theorem}\label{TH_step-and-Koenigs}
Let $\varphi$ and $\psi$ be commuting parabolic self-maps of the unit disc. The following three statements hold.
\begin{statlist}
  \item\label{IT_SaK-both-zero-st}
    If both $\varphi$ and $\psi$ are of zero hyperbolic step, then $h_\psi=c^{-1}_{\varphi,\psi}\,h_\varphi.$
  \item\label{IT_SaK-composition}
     If $\varphi$ is of zero hyperbolic step, then $\varphi\circ\psi$ is also of zero hyperbolic step and different from~$\id_\UD${\rm;} in particular, ${h_{\varphi\circ\psi}=c^{-1}_{\varphi,\varphi\circ\psi}\,h_\varphi}.$
  \item\label{IT_SaK-real-SLC}
   Suppose that system~\eqref{EQ_SLsystem} admits a solution ${h\in\Hol(\UD,\C)}$ for some ${(c_1,c_2)\in\Real^2\setminus\{(0,0)\}}$. If~$\varphi$~is of zero hyperbolic step, then $\psi$ is also of zero hyperbolic step {\rm(}and hence ${h_\psi=c^{-1}_{\varphi,\psi}\,h_\varphi}${\rm)}.
\end{statlist}
\end{theorem}

\begin{remark}
The hypothesis that \textit{both} self-maps are of zero hyperbolic step cannot be removed from statement~\ref{IT_SaK-both-zero-st} of the above theorem. Indeed, consider the Riemann map~$h:\UD\to\Omega$ of
$$
  \Omega:=\C\,\setminus\,\bigcup_{k\in\mathbb Z}\{x+ik\colon x\le-1\},
$$
normalized by $h(0)=0$, $h'(0)>0$. Then $\varphi:={h^{-1}\circ(h+1)}$ and $\psi:={h^{-1}\circ(h+i)}$ are commuting holomorphic self-maps of~$\UD$. Moreover, $(\C, h, z\mapsto z+1)$ is a holomorphic model of $\varphi$, so that it is parabolic of zero hyperbolic step with ${h_\varphi=h}$, and ${h_\varphi\circ\psi}={h_\varphi+i}$. However, $\psi$ is a parabolic automorphism of~$\UD$ and therefore, $h_\psi$ is a conformal map of~$\UD$ onto a half-plane and so the equality $i h_\psi=h_\varphi$ is not possible. Note that being a parabolic automorphism, $\psi$~is of positive parabolic step.
\end{remark}

\addtocontents{toc}{\SkipTocEntry}\subsection{Proofs} We start by establishing two auxiliary statements.

\begin{lemma}\label{lema_constante} Let $\Phi\in\Hol(\H_{\mathrm R})$ be a parabolic self-map of the right half-plane. For ${0<r<1}$, let
	$$
	V_\Phi(r):=\{w\in\H_{\mathrm R}: \phd_{\H_{\mathrm R}}(w,\Phi(w))<r\},
	$$
	where $\phd_{\H_{\mathrm R}}$ is the pseudohyperbolic distance in $\H_{\mathrm R}$. Then, for any $\xi\in V_\Phi(1/9)$ and any $w\in\D$, it holds that
	$$
	\xi+\frac15 w\Re\xi~\in~ V_\Phi(1/3).
	$$
\end{lemma}
\begin{proof} Fix $\xi\in V_\Phi(1/9)$ and $w\in\D$ and denote $u:=\xi+\frac15 w\Re\xi$. First of all, we note that
	$$
	\Re u=(\Re\xi)\big(1+(1/5)\Re w\big)>0,
	$$
	thus $u\in\H_{\mathrm R}$. With the help of the triangle inequality and the Schwartz\,--\,Pick Lemma, we obtain
	\begin{eqnarray*}
	\phd_{\H_{\mathrm R}}(u,\Phi(u)) &\leq&
          \phd_{\H_{\mathrm R}}(u,\xi) + \phd_{\H_{\mathrm R}}(\xi,\Phi(\xi))
                                              + \phd_{\H_{\mathrm R}}(\Phi(\xi),\Phi(u))\\
          &\leq&
 	2\phd_{\H_{\mathrm R}}(u,\xi)+\phd_{\H_{\mathrm R}}(\xi,\Phi(\xi))<1/9+2\phd_{\H_{\mathrm R}}(u,\xi).
	\end{eqnarray*}
	Therefore, it remains to check that $\phd_{\H_{\mathrm R}}(u,\xi)\leq 1/9$, which is quite straightforward:
	\begin{equation*}
	\phd_{\H_{\mathrm R}}(u,\xi) =\left|\frac{\xi-u}{\xi+\overline{u}}\right|=\frac{(1/5)|w|\Re\xi}{\Re\xi|2+(1/5)\overline{w}|}\leq \frac{1/5}{2-1/5}=\frac19.\qedhere
	\end{equation*}	
\end{proof}
\medskip

\begin{proposition}\label{positivo2} 
Let $\varphi\in\Hol(\UD)$ be a parabolic self-map with Denjoy\,--\,Wolff point ${\tau\in\partial\D}$ and Koenigs function $h_{\varphi}$. Let $\psi\in\Zn(\varphi)\setminus\{\id_\D\}$.
Then
	$$
	\angle \lim_{z\to\tau}\dfrac{h_\varphi(\psi(z))-h_\varphi(z)}{h^\prime_\varphi(z)(\psi(z)-z)}=1.
	$$
\end{proposition}
\begin{proof} Consider $\Psi:=C\circ\psi\circ C^{-1}$, where $C$ is the usual Cayley map with a pole at~$\tau$ and mapping~$\UD$ onto the right half-plane~$\H_{\mathrm R}$; i.e. ${C(z)=\frac{\tau+z}{\tau-z}}$, ${z\in\D}$. {By results of Behan and Cowen, see Section~\ref{SS_commuting}, $\psi$ is a parabolic self-map with Denjoy\,--\,Wolff point at~$\tau$. Therefore, $\Psi$ and $\Phi:={C\circ\varphi\circ C^{-1}}$ are parabolic self-maps of~$\H_{\mathrm R}$ with Denjoy\,--\,Wolff point at~$\infty$. In particular,
	\begin{equation}\label{Eq:positivo2}
		\angle\lim_{w\to\infty}\frac{\Psi(w)}w=1.
	\end{equation}

Fix for a while some ${z\in\UD}$. Denote ${\xi:=C(z)}$ and let $H:=h_\varphi\circ C^{-1}$}. Then
\begin{align*}
&h_\varphi(\psi(z))-h_\varphi(z)=H(\Psi(\xi))-H(\xi),\\[1ex]
&h_\varphi^\prime(z)=H^\prime\big(C(z)\big)C^\prime(z)=H^\prime(\xi)\frac{2\tau}{\big(\tau-C^{-1}(\xi)\big)^2}= H^\prime(\xi)\frac{(1+\xi)^2}{2\tau},\quad\text{and}\\[1ex]
&\psi(z)-z=C^{-1}(\Psi(\xi))-C^{-1}(\xi)= 2\tau\frac{\Psi(\xi)-\xi}{(1+\xi)(1+\Psi(\xi))}.
\end{align*}

\medskip\noindent Bearing in mind~\eqref{Eq:positivo2}, it is therefore sufficient to prove that
\begin{equation}\label{EQ_half-plane-version}
 \angle \lim_{\xi\to\infty}\frac{H(\Psi(\xi))-H(\xi)}{H^\prime(\xi)(\Psi(\xi)-\xi)}=1.
\end{equation}

To this end, set $c:=1/5$ and consider the family of functions
	$$
	g_\xi(w):=\frac{H(\xi+cw\Re\xi)-H(\xi)}{c\,H^\prime(\xi)\Re\xi},\quad  w\in \D,\quad \xi\in V_\Phi(1/9).
	$$
By Theorem~\ref{Thm:V-sets}\+\ref{IT_V-sets(4)}, we know that $h_\varphi$ is univalent in $V_\varphi(1/3)$. Hence, $H=h_\varphi\circ C^{-1}$ is univalent in
	$$
C\big(V_\varphi(1/3)\big)=V_\Phi(1/3), \quad V_\Phi(r):=\big\{w\in\H_{\mathrm R}:\phd_{\H_{\mathrm R}}(w,\Phi(w))<r\big\}.
	$$
	Therefore, using Lemma~\ref{lema_constante}, we may deduce that for every~${\xi\in V_\Phi(1/9)}$, the function $g_\xi$ is well-defined, holomorphic and moreover, univalent in~$\UD$. Furthermore, $g_\xi(0)={g_\xi^\prime(0)-1=0}$.

Now, we claim that
	\begin{equation}\label{EQ_claim}
		|g_\xi(w)-w|\leq 54|w|^2\quad  \text{when~$~|w|<1/2$.}
	\end{equation}
Assume for a moment that~\eqref{EQ_claim} holds and consider
	an arbitrary sequence $(\xi_n)\subset\H_{\mathrm R}$ converging non-tangentially to $\infty$. Denote ${w_n:=\big(\Psi(\xi_n)-\xi_n\big)/(c\Re \xi_n)}$. By \eqref{Eq:positivo2},
	\begin{equation}\label{EQ_to0}
		|w_n|=\left|\frac{\Psi(\xi_n)}{\xi_n}-1\right|\,\cdot\, c^{-1}\cdot\,\left|\frac{\xi_n}
		{\Re \xi_n}\right|
		~\longrightarrow~0\quad\text{as~$~n\to+\infty$}.
	\end{equation}
	In particular, omitting a finite number of terms, we may assume that  ${|w_{n}|<1/2}$ for all ${n\in\Natural}$. Moreover, recall that $\Psi$ is non-elliptic. Hence, ${w_{n}\neq0}$ for all ${n\in\Natural}$. Finally, by Theorem~\ref{Thm:V-sets}\+\ref{IT_V-sets(3)} the domain~$V_\varphi(1/9)$ is isogonal at~$\tau$. As a consequence, ${\xi_n\in V_\Phi(1/9)}$ for all~$n$ large enough. Thus,
	using~\eqref{EQ_claim} and~\eqref{EQ_to0}, we may conclude that
	$$ \lim_{n\to+\infty}
\frac{H\big(\Psi(\xi_n)\big)-H(\xi_n)}{H^\prime(\xi_n)\big(\Psi(\xi_n)-\xi_n\big)}
=\lim_{n\to+\infty}\frac{g_{\xi_n}(w_n)}{w_n}=1.
	$$

	It remains to prove~\eqref{EQ_claim}. Fix an arbitrary~$w$ with $|w|<1/2$. Using Cauchy's integral formula with the contour $\Gamma:=\{z:|z|=2/3\}$ oriented counterclockwise and the Growth Theorem, see e.g. \cite[Theorem~2.6 on p.\,33]{Duren}, we find that
	\begin{align*}
		\big|g_{\xi}(w)-w\big| &=\big|g_{\xi}(w)-g_{\xi}(0)-g_{\xi}^{\prime }(0)w\big|=\displaystyle\left\vert \frac{1}{2\pi i}\int_\Gamma g_{\xi}(z)\Big( \frac{1
		}{z-w}-\frac{1}{z }-\frac{w}{z^{2}}\Big) \di z \right\vert  \\
		&=\displaystyle\left\vert \frac{1}{2\pi i}\int_\Gamma g_{\xi}(z)\frac{w^{2}}{%
			z^{2}(z-w)}\di z \right\vert  \leq \displaystyle\frac{1}{2\pi }\int_\Gamma\frac{|z|}{(1-|z|)^{2}}\,
		\frac{|w|^{2}}{|z|^{2}\,|z-w|}\,|\di z |\leq 54|w|^{2}. \qedhere
	\end{align*}	
\end{proof}\medskip

\begin{proof}[\proofof{Theorem~\ref{TH_SLC(2formulas)}}]
\color{black} Recall that by Theorem~\ref{Thm:zero-hyperbolic-step}\+\ref{IT_THzero-SimLin}, ${h_\varphi\circ\psi-h_\varphi}$ is a constant, and that by the very definition, this constant is the simultaneous linearization coefficient~$c_{\varphi,\psi}$. In other words,
\begin{equation}\label{EQ_c-psi-phi}
c_{\varphi,\psi}=h_\varphi(\psi (z))-h_\varphi (z)\quad\text{ for all $~{z\in \UD}$}.
\end{equation}
Throughout the proof we suppose that ${\psi\neq\id_\UD}$, because for ${\psi=\id_\UD}$ both statements of Theorem~\ref{TH_SLC(2formulas)} are trivial.

\StepP{\ref{IT_SLC1}}
By \eqref{EQ_c-psi-phi} and Proposition~\ref{positivo2},
	$$
	\angle \lim_{z\to\tau}\dfrac{c_{\varphi,\psi}}{h^\prime_\varphi(z)(\psi(z)-z)}=\angle \lim_{z\to\tau}\dfrac{h_\varphi(\psi (z))-h_\varphi (z)}{h^\prime_\varphi(z)(\psi(z)-z)}=1,
	$$
	or equivalently
    \begin{equation}\label{EQ_another-formula-for-SLC}
      c_{\varphi,\psi} =\angle \lim_{z\to\tau}h_\varphi^\prime(z)\big(\psi(z)-z\big).
    \end{equation}

By the same argument applied to~$\varphi$ instead of~$\psi$, we have
\begin{equation}\label{EQ_another-formula-bis}
  1=c_{\varphi,\varphi}=\angle \lim_{z\to\tau}h_\varphi^\prime(z)\big(\varphi(z)-z\big).
\end{equation}
Combining~\eqref{EQ_another-formula-for-SLC} with~\eqref{EQ_another-formula-bis} immediately yields the desired formula
$$
  c_{\varphi,\psi}=\angle \lim_{z\to\tau}\frac{\psi(z)-z}{\varphi(z)-z}.
$$

\StepP{\ref{IT_SLC2}} As mentioned above, we may assume that $\psi \neq \id_\UD$. By the hypothesis,
\begin{equation}\label{EQ_double}
	h\circ \varphi=h+c_1,\qquad h\circ \psi=h+c_2,
\end{equation}
	for some holomorphic function $h:\UD\to\C$ and  some $(c_1,c_2)\in\C^2\setminus\{(0,0)\}$. We have to show that  ${c_1\neq0}$ and ${c_2/c_1=c_{\varphi,\psi}}$.

Note that since $\varphi\not\in\Aut$  and  $\psi\neq\id_\UD$, in view of Proposition~\ref{Prop:firstpropertiescoefficient}, we have $c_{\varphi,\psi}\in{\Complex\setminus(-\infty,0]}$.
In particular, if $c_{\varphi,\psi}\in\mathbb Q$, then ${c_{\varphi,\psi}=m/n}$ for some ${m,n\in\Natural}$, and consequently, by Proposition~\ref{Prop:firstpropertiescoefficient}\+\ref{firstpropereties2}, we have ${\psi^{\circ n}=\varphi^{\circ m}}$.  Using~\eqref{EQ_double}, we therefore obtain
$$
  h+mc_1=h\circ\varphi^{\circ m}=h\circ\psi^{\circ n}=h+nc_2.
$$
Thus, $c_2=mc_1/n=c_{\varphi,\psi}c_1$. This completes the proof in case ${c_{\varphi,\psi}\in\mathbb Q}$.

Suppose now that $c_{\varphi,\psi}\not\in\mathbb Q$. The first equality in~\eqref{EQ_double} implies that $(S',h,z\mapsto z+c_1)$, where $S':=\bigcup_{n\in\Natural} \big(h(\UD)-nc_1\big)$, is a holomorphic semimodel for~$\varphi$. Therefore, appealing to Lemma~\ref{Le:morphism}, we see that ${h=\beta\circ h_\varphi}$ for some~$\beta\in\Hol(\C)$. Substitute this representation for~$h$ in~\eqref{EQ_double} and recall the relations ${h_\varphi\circ\varphi=h_\varphi+1}$, ${h_\varphi\circ\psi=h_\varphi+c_{\varphi,\psi}}$. In this way, using the identity principle for holomorphic functions,  we easily obtain
\begin{equation}\label{EQ-beta-period}
\beta(w+1)=\beta(w)+c_1,\quad \beta(w+c_{\varphi,\psi})=\beta(w)+c_2\quad\text{for all~$~w\in\C$.}
\end{equation}
In particular, it follows that the derivative of~$\beta$ is periodic with two periods ${\omega_1=1}$ and ${\omega_2=c_{\varphi,\psi}}$. Since ${\omega_2/\omega_1}$ is not a rational real number, it follows that $\beta'$ is either constant or it is a non-constant function with at least two primitive periods. By basic results in the theory of elliptic functions, see e.g. \cite[pp.~2 and~8]{Akh}, a non-constant entire function can have at most one primitive period. Consequently, $\beta'$ is constant. Thus, $\beta(w)={aw+b}$ for all ${w\in\C}$ and some ${a,b\in\C}$. Recalling~\eqref{EQ-beta-period}, we have ${a=c_1}$ and ${ac_{\varphi,\psi}=c_2}$, which immediately yields the desired conclusion.
\end{proof}

\begin{proof}[\proofof{Theorem~\ref{TH_step-and-Koenigs}\+\ref{IT_SaK-both-zero-st}}]
Let $\varphi,\psi\in\Hol(\UD)$ be two parabolic self-maps of zero hyperbolic step. Suppose that ${\varphi\circ\psi=\psi\circ\varphi}$. We have to show that ${h_\psi=c_{\varphi,\psi}^{-1}h_\varphi}$.

By Theorem~\ref{Thm:zero-hyperbolic-step}\+\ref{IT_THzero-SimLin},
\begin{equation}\label{EQ_sys}
  h_\varphi\circ\psi = h_\varphi + c_{\varphi,\psi},\qquad
  h_\varphi\circ\varphi = h_\varphi + 1.
\end{equation}
Consider two cases.

\StepC1{$c_{\varphi,\psi}\not\in\mathbb Q$} Taking into account~\eqref{EQ_sys}, by Theorem~\ref{TH_SLC(2formulas)}\+\ref{IT_SLC2} applied with $\varphi$ and $\psi$ interchanged, we have $c_{\psi,\varphi}={1/c_{\varphi,\psi}\not\in\mathbb Q}$. Therefore, by Proposition~\ref{PR_SLC(cont)} for ${h:=h_\varphi}$, ${(c_1,c_2)}:={(c_{\varphi,\psi},1)}$ and for $\varphi$ and $\psi$ again interchanged, we have ${h_\varphi-h_\varphi(0)=c_{\varphi,\psi}h_\psi}.$ It remains to recall that according to the normalization of the Koenigs function adopted in this paper,  ${h_\varphi(0)=0}$; see Remark~\ref{RM_normalization}.

\StepC2{$c_{\varphi,\psi}=m/n\in\mathbb Q$} A parabolic self-map of zero hyperbolic step cannot be an automorphism of~$\UD$. Hence, recalling~\eqref{EQ_sys} and appealing to Proposition~\ref{Prop:firstpropertiescoefficient}, we conclude that ${c_{\varphi,\psi}>0}$ and that $f:={\psi^{\circ n}=\varphi^{\circ m}}$. From the definition (and uniqueness) of the Koenigs function it then easily follows that ${(1/n)h_\psi=h_f=(1/m)h_{\varphi}}$. Thus, $h_\varphi=(m/n)h_\psi$ as desired.
\end{proof}

\begin{proof}[\proofof{Theorem~\ref{TH_step-and-Koenigs}\+\ref{IT_SaK-composition}}]
Let $\varphi$ and $\psi$ be commuting parabolic self-maps of~$\UD$. We have to show that if~$\varphi$~is of zero hyperbolic step, then so is $\varphi\circ\psi$ and that ${\varphi\circ\psi\neq\id_\UD}$. The latter statement follows immediately from the fact that $\varphi$ is not an automorphism of~$\UD$.  Clearly, ${\varphi\circ\psi}$ commutes with~$\varphi$.
Furthermore, fixing ${z_0\in\D}$ and applying the Schwarz\,--\,Pick Lemma to the self-map $\psi^{\circ n}$, we have
\begin{eqnarray}
	\hd_\D\big((\varphi\circ\psi)^{\circ(n+1)}(z_0),(\varphi\circ\psi)^{\circ n}(z_0)\big)\notag
      &=&
      \hd_\D\big(\psi^{\circ n}(\varphi^{\circ n}\circ\varphi\circ\psi(z_0)),
                                    \psi^{\circ n}(\varphi^{\circ n}(z_0))\big)\\
      &\leq&
	\hd_\D\big(\varphi^{\circ n}(\varphi\circ\psi(z_0)),\varphi^{\circ n}(z_0)\big)\qquad\text{for all~$~n\in\Natural$.}\notag
\end{eqnarray}
	The latter quantity tends to zero as $n\to+\infty$ because $\varphi$ is of zero hyperbolic step, see e.g. \cite[Corollary~4.6.9\+(iv) on p.\,246]{Abate2}. Thus, ${\varphi\circ\psi}$ is of zero hyperbolic step, as desired.
\end{proof}

\begin{proof}[\proofof{Theorem~\ref{TH_step-and-Koenigs}\+\ref{IT_SaK-real-SLC}}]
 Let $\varphi,\psi\in\Hol(\UD)$ be two commuting parabolic self-maps. Suppose that $\varphi$ is of zero hyperbolic step and that there exists $h\in\Hol(\UD,\C)$ such that
\begin{equation}\label{EQ_sysh}
  h\circ\varphi = h + c_1,\qquad
  h\circ\psi = h + c_2,
\end{equation} for some $(c_1,c_2)\in\Real^2\setminus\{(0,0)\}$. We have to show that under these hypotheses, $\psi$ is of zero hyperbolic step.

By Theorem~\ref{TH_SLC(2formulas)}\+\ref{IT_SLC2},  ${c_{\varphi,\psi}\in\Real}$. Moreover, Proposition~\ref{Prop:firstpropertiescoefficient} allows us to conclude that ${c_{\varphi,\psi}>0}$.
By Theorem~\ref{Thm:V-sets}, the set $V:=V_\varphi(1/3)$ is a simply connected $\varphi$-absorbing domain. Moreover, ${\psi(V)\subset V}$ by Lemma~\ref{Le:V-sets}\+\ref{V-sets-fact1}. Consider any conformal map $f$ of~$\UD$ onto~$V$. According to Theorem~\ref{Thm:zero-hyperbolic-step}\+\ref{IT_THzero-SimLin}, ${h_\varphi\circ\psi}={h_\varphi+c_{\varphi,\psi}}$. It follows that
\begin{equation}\label{EQ_h-star-psi-tilde}
h_*\circ\mtilde\psi^{\circ n}=h_*+nc_{\varphi,\psi}\quad\text{for all~$~n\in\Natural$},
\end{equation}
where $h_*:=h_\varphi\circ f$ and $\mtilde\psi\in\Hol(\UD)$ is defined by $\mtilde\psi:={f^{-1}\circ\psi\circ f}$.
To prove that $\psi$ is of zero hyperbolic step, it is sufficient to check that~$\mtilde\psi$ is of zero hyperbolic step. Indeed, fixing a point ${\zeta_0\in\UD}$ and denoting $z_0:=f(\zeta_0)$, we have
$$
  \hd_\UD\big(\psi^{\circ(n+1)}(z_0),\psi^{\circ n}(z_0)\big)\le \hd_V\big(\psi^{\circ(n+1)}(z_0),\psi^{\circ n}(z_0)\big)=
  \hd_\UD\big(\mtilde\psi^{\circ(n+1)}(\zeta_0),\mtilde\psi^{\circ n}(\zeta_0)\big),
$$
which would force $\hd_\UD\big(\psi^{\circ(n+1)}(z_0),\psi^{\circ n}(z_0)\big)\to0$ as ${n\to+\infty}$ if we are able to show that $\mtilde\psi$ is of zero hyperbolic step.

Now let $\zeta_0\in\UD$. By Theorem~\ref{Thm:V-sets}\+\ref{IT_V-sets(4)}, $h_\varphi$ is univalent in~$V$. Hence, $h_*$ is univalent in~$\UD$. Taking this into account and using~\eqref{EQ_h-star-psi-tilde}, we obtain
$$
\mathrm{D}_h\mtilde\psi^{\circ n}(\zeta_0):=\frac{1-|\zeta_0|^2}{1-|\mtilde\psi^{\circ n}(\zeta_0)|^2}\,
  \big|\big(\mtilde\psi^{\circ n}\big)'(\zeta_0)\big|\,=\,\frac{H(\zeta_0)}{H(\mtilde\psi^{\circ n}(\zeta_0))}\neq0
$$
for all $n\in\Natural$, where $H(\zeta):=|h'_*(\zeta)|(1-|\zeta|^2)$, ${\zeta\in\UD}$. Since $\varphi$ is of zero hyperbolic step, the base space of the canonical holomorphic model for~$\varphi$ is the whole complex plane~$\C$. Therefore, recalling that $c_{\varphi,\psi}>0$  we conclude, with the help of Remark~\ref{Re:almostasemigroup}, that for any ${R>0}$ there exists $n_0=n_0(R)$ such that for all $n\ge n_0$,
$$
  K_n:=\big\{w:|w-h_*\big(\mtilde\psi^{\circ n}(\zeta_0)\big)|\le R\big\}
      \,=\,\big\{w:|w-h_*(\zeta_0)|\le R\big\} + n c_{\varphi,\psi}~\subset~h_\varphi(V)=h_*(\UD).
$$
Since $h_*$ is univalent, it follows (see e.g. \cite[Corollary~1.4]{Pommerenke:BB}) that $H(\mtilde\psi^{\circ n}(\zeta_0))> R$ for all ${n\in\Natural}$ large enough. With $R>0$ being in this argument arbitrary, it follows that ${H(\mtilde\psi^{\circ n}(\zeta_0))\to+\infty}$ and hence, ${\mathrm{D}_h\mtilde\psi^{\circ n}(\zeta_0)\to0}$ as ${n\to+\infty}$. Thus, Proposition~\ref{Pro:step-derivative}, proved in the Appendix, implies that $\mtilde\psi$ is of zero hyperbolic step, as desired.
\end{proof}

\section{Appendix: the hyperbolic step and derivatives of the iterates}\label{S_APP}
Below we give a necessary and sufficient condition for a non-elliptic self-map to be of zero hyperbolic step, which we used in the proof of Theorem~\ref{TH_step-and-Koenigs}\+\ref{IT_SaK-real-SLC} and which would be probably also useful when one tries to apply the results established in the previous sections.
On the one hand, this condition does not seem to be unexpected for specialists, but on the other hand, we could not find it in the literature. By this reason, we provide a detailed proof.
\begin{proposition}\label{Pro:step-derivative}
Let $\varphi\in\Hol(\UD)$ be a non-elliptic self-map. Then the following conditions are equivalent:
\begin{ourlist}
  \item\label{IT_step-derivative:z.h.st.}
      $\varphi$ is parabolic of zero hyperbolic step;
  \item\label{IT_step-derivative:exists-z}
      there exists $z_0\in \D$  with $(\varphi^{\circ n})'(z_0)\neq 0$ for all ${n\in \N}$  such that
      \begin{equation}\label{Eq:step-derivative6}
        \lim_{n\to+\infty}\frac{(\varphi^{\circ n})'(z_0)}{1-|\varphi^{\circ n}(z_0)|^{2}}=0;
      \end{equation}
  \item\label{IT_step-derivative:all-z}
     equality~\eqref{Eq:step-derivative6} holds for all~$z_0\in\UD$.
\end{ourlist}
\end{proposition}

\begin{remark}
 It is worth mentioning, even though it is just a trivial remark,  that condition~\eqref{Eq:step-derivative6} can be stated in terms of the so-called \emph{hyperbolic distortion} $\mathrm D_h\varphi(z):=\lambda_\UD(\varphi(z))|\varphi'(z)|/\lambda(z)$, where $\lambda(z):=1/(1-|z|^2)$ is the density of the hyperbolic metric in~$\UD$. Clearly, \eqref{Eq:step-derivative6} is equivalent to ${\mathrm D_h\varphi^{\circ n}(z_0)\to0}$ as~${n\to+\infty}$. Properties of the hyperbolic distortion seem to be deeply linked to the dynamical behaviour of a holomorphic self-map, see e.g. \cite{GuKouMouRoth}.
\end{remark}

\begin{remark}
An analogue of Proposition~\ref{Pro:step-derivative} is trivially true for elliptic self-maps. More precisely, if
${\varphi\in\Hol(\UD,\UD)\setminus\Aut}$ is elliptic, then it is automatically of zero hyperbolic step, and \eqref{Eq:step-derivative6} holds for all~$z_0\in\UD$ because $\varphi^{\circ n}\to\const\in\UD$ locally uniformly in~$\UD$ as ${n\to+\infty}$. In contrast, if ${\varphi\in\Aut}$, then  ${\mathrm D_h\varphi^{\circ n}(z)=1}$ for all ${n\in\Natural}$ and all~$z\in\UD$, and as a result the limit in~\eqref{Eq:step-derivative6} is strictly positive for all~${z_0\in\UD}$.
\end{remark}

\begin{proof}[\proofof{Proposition~\ref{Pro:step-derivative}}]
If $\varphi\in\Aut$, then the statement of the proposition holds trivially because none of the conditions \ref{IT_step-derivative:z.h.st.}\,--\,\ref{IT_step-derivative:all-z} is satisfied for automorphisms.

So for the rest of the proof we suppose that ${\varphi\not\in\Aut}$.
Clearly, \ref{IT_step-derivative:all-z}$\,\Rightarrow\,$\ref{IT_step-derivative:exists-z}.

Fix an arbitrary point~${z_0\in\UD}$. If ${(\varphi^{\circ n})'(z_0)=0}$ for some ${n\in \N}$, then~\eqref{Eq:step-derivative6} holds trivially. So we assume that ${(\varphi^{\circ n})'(z_0)\neq0}$ for all ${n\in \N}$. In order to make the proof of the implications \hbox{\ref{IT_step-derivative:z.h.st.}$\,\Rightarrow\,$\ref{IT_step-derivative:all-z}} and \hbox{\ref{IT_step-derivative:exists-z}$\,\Rightarrow\,$\ref{IT_step-derivative:z.h.st.}} less technical, we are going to replace the autonomous dynamical system in~$\UD$ induced by the iterates of~$\varphi$ with a non-autonomous dynamical system in~$\UD$, with the advantage that to the orbit $\big(z_{n}:= \varphi^{\circ n}(z_{0})\big)_{n\in \N}$  there would correspond the \textit{fixed point} of the new dynamical system at the origin. To this end, we employ the time-dependent change of variables in~$\UD$ given by the automorphisms $L_{n}(\zeta):=\frac{z_{n}-\zeta}{1-\overline{z_{n}}\zeta}$, ${\zeta\in \D}$,  and define $g_{n}:={L_n^{-1}\circ\varphi\circ L_{n-1}}$, ${n\in\Natural}$.
It is not difficult to see that $g_n$'s are holomorphic self-maps of~$\UD$ with a fixed point at the origin and that
\begin{equation}\label{EQ_from-g_n-to-varphi}
\varphi^{\circ n} = {L_{n}\circ g_{n}\circ g_{n-1}\circ \dots\circ g_{1}\circ L_0^{-1}},
   \quad\text{for all $~{n\in\Natural}$.}
\end{equation}
Now we consider the sequence~${(\zeta_n)\subset\UD}$ defined recursively by
$$
  \zeta_{n}:=g_{n}(\zeta_{n-1})\quad\text{for all $~{n\in\Natural}$},\qquad \zeta_{0}:=L_0^{-1}(z_{1}).
$$
Using~\eqref{EQ_from-g_n-to-varphi} we see that $L_{n}(\zeta_{n})=z_{n+1}$ for all ${n\in\Natural\cup\{0\}}$. Recalling that $\varphi$ has no fixed point in~$\D$ by the hypothesis, we have ${L_n(\zeta_n)=z_{n+1}\neq z_n=L_n(0)}$ and therefore, ${\zeta_{n}\neq 0}$ for any  ${n\in\Natural\cup\{0\}}$.

By the invariance of the hyperbolic distance~$\hd_\UD$ under automorphisms of~$\UD$, we have
\begin{equation}\label{EQ_hset-from-z-to-zeta}
\hd_{\D}\big(\varphi^{\circ (n+1)}(z_{0}), \varphi^{\circ n}(z_{0})\big)
      =\hd_{\D}(z_{n+1}, z_n)
     =\hd_{\D}\big(L_{n}(\zeta_n),\,
                  L_{n}(0) \big)=\hd_{\D}(\zeta_{n}, 0 ).
\end{equation}
Therefore, $\varphi$ is parabolic of zero hyperbolic step if and only if the sequence $(\zeta_{n})$ converges to~$0$.

Given $m>n$, consider the function $G_{n,m}:\D\to \D$ defined by
$$
  G_{n,m}:={g_{m}\circ g_{m-1}\circ \dots\circ g_{n+1}}={L_m^{-1}\circ \varphi^{\circ (m-n)}\circ L_n}.
$$
Notice that
\begin{equation}\label{Eq:step-derivative3}
G_{0,n}(0)=0\quad\text{and}\quad
G'_{0,n}(0)= \frac{(\varphi^{\circ n})'(z_{0})}{1-|\varphi^{\circ n}(z_{0})|^{2}}(1-|z_{0}|^{2})
\end{equation}
for all $n\in\Natural$.
Fixed $m>n$, by the Schwarz Lemma we have
\begin{equation}\label{Eq:step-derivative2}
|G'_{0,m}(0)|=|G'_{0,n}(0)G'_{n,m}(0)|\leq |G'_{n,m}(0)|.
\end{equation}
Moreover, since $G_{n,m}\not\in\Aut$, we can apply
the Schwarz\,--\,Pick's Lemma to the function
$$
F_{n,m}(\zeta):=
\left\{
\begin{array}{ll}
G_{n,m}(\zeta)/\zeta, & \quad \textrm{if } \ \zeta\in \D\setminus\{0\},\\[.4ex]
G_{n,m}'(0), & \quad \textrm{if } \ \zeta=0,
\end{array}
\right.
$$
to obtain
\begin{equation}\label{Eq:step-derivative}
\begin{split}
\left|G_{n,m}'(0)-\frac{\zeta_{m}}{\zeta_{n}}\right| &=\big|F_{n,m}(0)-F_{n,m}(\zeta_{n})\big|\\
&= \big|1-\overline{F_{n,m}(0)}F_{n,m}(\zeta_{n})\big|\,
 \phd_\UD\big(F_{n,m}(0),\,F_{n,m}(\zeta_{n})\big) \\[.55ex]
&\leq  \big|1-\overline{F_{n,m}(0)}F_{n,m}(\zeta_{n})\big| \, |\zeta_n |~\leq~2|\zeta_{n}|.
\end{split}
\end{equation}

It follows that
\begin{equation}\label{Eq:step-derivative4}
|G'_{0,m}(0)|\leq |G'_{n,m}(0)|\leq \left|\frac{\zeta_{m}}{\zeta_{n}}\right| +2|\zeta_{n}|.
\end{equation}

Assume now that $\varphi$ is parabolic of zero hyperbolic step. Then $(\zeta_n)$ converges to~$0$. Fix an ${\varepsilon>0}$  and choose some ${n\in\Natural}$ such that~$|\zeta_n|<\varepsilon/3$. Then by~\eqref{Eq:step-derivative4}, we have ${|G'_{0,m}(0)|<\varepsilon}$ for all ${m\in\Natural}$ large enough. Since ${\varepsilon>0}$ here is arbitrary, it follows that ${|G'_{0,m}(0)|\to0}$ as ${m\to+\infty}$ and hence, in view of~\eqref{Eq:step-derivative3}, equality~\eqref{Eq:step-derivative6} holds. This proves the implication~\hbox{\ref{IT_step-derivative:z.h.st.}$\,\Rightarrow\,$\ref{IT_step-derivative:all-z}}.

Conversely, assume that there exists $z_0\in\UD$ satisfying ${(\varphi^{\circ n})'(z_0)\neq 0}$ for all ${n\in \N}$ and such that~\eqref{Eq:step-derivative6} holds.
Then ${G'_{0,m}(0)\to0}$ as ${m\to+\infty}$.
By \eqref{Eq:step-derivative3}, $G'_{0,n}(0)\neq 0$ for all $n\in\Natural$. With the help of~\eqref{Eq:step-derivative2}, it follows that for each fixed ${n\in\Natural}$, we have ${G'_{n,m}(0)\to0}$ as ${m\to+\infty}$. Since by the Schwarz Lemma, the  sequence $\big(r_n:=|\zeta_n|\big)$ does not increase, it has a limit $\ell\in [0,1)$.  We wish to prove that ${\ell=0}$. Let us recall that ${\zeta_n\neq0}$ for all ${n\in\Natural}$. According to~\eqref{Eq:step-derivative},
$$
\left|G_{n,m}'(0)-\frac{\zeta_{m}}{\zeta_{n}}\right|
    \leq \big|1-\overline{G_{n,m}'(0)}F_{n,m}(\zeta_{n})\big| \, r_n.
$$
Passing in this inequality to the limit as ${m\to+\infty}$ with $n$ fixed, we deduce that ${\ell/r_n\leq r_n}$. That is, ${\ell\leq r_n^2}$ for all ${n\in\Natural}$. Passing now to the limit as ${n\to+\infty}$, we finally get that $\ell^2\geq \ell\in [0,1)$. Thus $\ell=0$, which according to~\eqref{EQ_hset-from-z-to-zeta} means that $\varphi$ has zero hyperbolic step. To complete the proof of the implication \hbox{\ref{IT_step-derivative:exists-z}$\,\Rightarrow\,$\ref{IT_step-derivative:z.h.st.}}, it remains to recall that a non-elliptic holomorphic self-map of zero hyperbolic step is necessarily parabolic.
\end{proof}

As a corollary of Proposition~\ref{Pro:step-derivative}, we  obtain a necessary and sufficient condition of zero hyperbolic step for continuous one-parameter semigroups. Recall that
 a family $(\phi_{t})_{t\geq 0}$  of holomorphic self-maps $\phi_{t}:\D\to \D$ is called a \dff{one-parameter semigroup}\label{PAGE_one-paremeter-semigroup} if it verifies the following two algebraic properties:
\begin{enumerate}
\item[(i)] $\phi_{0}=\id_{\D}$;
\item[(ii)] $\phi_{t+s}=\phi_{t} \circ \phi_{s}$ for every $t,s\geq 0$.
\end{enumerate}
A one-parameter semigroup $(\phi_t)$ is said to be \dff{continuous} if ${\phi_{t}\to \phi_{0}=\id_\UD}$ locally uniformly in~$\D$ as ${t\to 0^{+}}$. Finally, a one-parameter semigroup~$(\phi_t)$ is called \dff{non-trivial} if ${\phi_t\neq\id_\UD}$ for at least one value of~${t>0}$.

By a classical result of Berkson and Porta~\cite{BP}, see also \cite[\S5.4]{Abate2}, \cite[\S10.1]{BCD-Book} or \cite{BCD:Koenigs}, every continuous one-parameter semigroup~$(\phi_t)$ is differentiable in~$t$ and admits an \dff{infinitesimal generator}, i.e. a holomorphic function ${G:\UD\to\C}$ such that
\begin{equation}\label{EQ_inf-gen}
\frac{\di \phi_t(z)}{\di t}=G\big(\phi_t(z)\big)\quad\text{ for all $~{t\ge0}~$ and all $~{z\in\UD}$.}
\end{equation}
Finally, it is known that all elements of a non-trivial continuous one-parameter semigroup different from~$\id_\UD$ have the same Denjoy\,--\,Wolff point; see e.g. \cite[\S5.5]{Abate2} or \cite[Theorem~8.3.1]{BCD-Book}. Therefore, it makes sense to talk about \dff{elliptic} and \dff{non-elliptic} continuous one-parameter semigroups, depending on whether the common Denjoy\,--\,Wolff point is contained in~$\UD$ or in~$\UC$. In a similar way, see e.g. \cite[\S8.3]{BCD-Book}, non-elliptic continuous one-parameter semigroups can be categorised into three types: hyperbolic semigroups, parabolic semigroups of positive hyperbolic step, and parabolic semigroups of zero hyperbolic step.

The necessary and sufficient condition for a non-elliptic one-parameter semigroup to have zero hyperbolic step, which we are going now to deduce from Proposition~\ref{Pro:step-derivative}, has a very simple geometric meaning: the hyperbolic norm of the velocity vector ${\lambda(\phi_t(z))\big|\tfrac{\di}{\di t}\phi_t(z)\big|}$ should tend to~$0$ as ${t\to+\infty}$. Here is the precise statement:
\begin{proposition}\label{Pro:step-derivative2}
Let $(\phi_t)$ be a non-elliptic continuous one-parameter semigroup in the unit disc and let $G$ be its infinitesimal generator.  Then the semigroup $(\phi_t)$  is parabolic of zero hyperbolic step if and only if
$$
\lim_{t\to +\infty}\frac{G(\phi_t(z))}{1-|\phi_t(z)|^{2}}=0
$$
for some~--- and hence for any~--- $z\in\UD$.
\end{proposition}
\begin{proof}
Since the function $\phi_1$ is univalent, see e.g. \cite[Theorem~8.1.17]{BCD-Book}, Proposition~\ref{Pro:step-derivative} implies that the semigroup $(\phi_t)$ is parabolic of zero hyperbolic step
if and only if
\begin{equation}\label{EQ_discrete-condition}
\lim_{\Natural\ni n\to+\infty}\frac{\phi_{n}'(z)}{1-|\phi_{n}(z)|^{2}}=0
\end{equation}
for some (and hence for all) ${z\in \D}$. Using the Schwarz\,--\,Pick Lemma, it is not difficult to see that for any ${z\in\UD}$, the function
$$
  F_z(t):=\frac{\phi_{t}'(z)}{1-|\phi_{t}(z)|^{2}},\quad t\in[0,+\infty),
$$
is non-increasing. Therefore,~\eqref{EQ_discrete-condition} is equivalent to ${F_z(t)\to0}$, as ${t\to+\infty}$. To complete the proof, it only remains to notice that $\phi_{t}'(z)={G(\phi_t(z))/G(z)}$ for all ${z\in \D}$ and all~${t\ge0}$. This equality can be easily deduced by combining~\eqref{EQ_inf-gen} with the PDE ${\partial \phi_t(z)/\partial t=G(z)\phi'_t(z)}$; see e.g. \cite[Proposition~10.1.8]{BCD-Book}.
\end{proof}

\end{document}